\documentclass[10pt]{article}
\usepackage{graphics}
\usepackage{graphicx}

\usepackage[a4paper, left=35mm,right=35mm,top=34mm,bottom=34mm]{geometry}
\usepackage[utf8]{inputenc}
\usepackage[T1]{fontenc}
\usepackage[english]{babel}

\usepackage{enumerate}
\usepackage{graphicx}
\usepackage{hyperref}
\hypersetup{
    colorlinks=true,
    linkcolor=blue,
    filecolor=magenta,      
    urlcolor=cyan,
}
\usepackage{lipsum}
\usepackage{amsfonts}
\usepackage{graphicx}
\usepackage{epstopdf}
\usepackage{algorithmic}
\usepackage{lipsum}
\usepackage{amsfonts}
\usepackage{graphicx}
\usepackage{epstopdf}

\usepackage{hyperref}
\ifpdf
  \DeclareGraphicsExtensions{.eps,.pdf,.png,.jpg}
\else
  \DeclareGraphicsExtensions{.eps}
\fi

\usepackage{xcolor}

\usepackage{tikz}

\usepackage{marginnote,cite,amsmath,bbm}

 \usepackage{dsfont}
 \usepackage{psfrag}
 \usepackage{color}
 \usepackage{tikz}
 \usepackage{subfigure}
\usepackage{mathtools,amsthm,amssymb,amsfonts}
\usepackage{algorithm}
\usepackage{algorithmic}
\makeatother
\theoremstyle{plain}
\def\del  {\partial}
\def\eps{\varepsilon}

\def\R{\mathbb{R}}

\def\N{\mathbb{N}}
\def\C{\mathbb{C}}

 \def\dx{{\rm d}x}

\newtheorem{proposition}{\textbf{Proposition}}
\newtheorem{corollary}{\textbf{Corollary}}
\newtheorem{remark}{\textbf{Remark}}
\newtheorem{theorem}{\textbf{Theorem}}
\newtheorem{lemma}{\textbf{Lemma}}
\newtheorem{definition}{\textbf{Definition}}

\usepackage{caption} 
\usepackage{fancyhdr}

\DeclareMathOperator{\diam}{diam}
\DeclareMathOperator{\loc}{loc}
\DeclareMathOperator{\ext}{Ext}
\DeclareMathOperator{\supp}{supp}

\author{
  {\normalsize Michael Hinz}\thanks{Fakult\"{a}t f\"{u}r Mathematik, Universit\"{a}t Bielefeld, Postfach 100131, 33501 Bielefeld,
Germany.}
		\and
  {\normalsize Anna Rozanova-Pierrat}\thanks{CentraleSup\'elec, Universit\'e Paris-Saclay, France.}
    \and
	{\normalsize Alexander Teplyaev}\thanks{Department of Mathematics, University of Connecticut, Storrs, CT 06269-3009 USA.}
		}
\title{Non-Lipschitz uniform domain shape optimization in linear acoustics}
\date{}

\begin{document}
\maketitle
\thispagestyle{fancy}

\begin{abstract}
\noindent We introduce new parametrized classes of shape admissible domains in $\mathbb{R}^n$, $n\geq 2$, and prove that they are compact with respect to the convergence in the sense of characteristic functions, the Hausdorff sense, the sense of compacts and the weak convergence of their boundary volumes. The domains in these classes are bounded $(\eps,\infty)$-domains with possibly fractal boundaries that can have parts of any non-uniform Hausdorff dimension greater or equal to $n-1$ and less than $n$. We prove the existence of optimal shapes in such classes for maximum energy dissipation in the framework of linear acoustics. A by-product 
 of our proof is the result that the class of bounded $(\eps,\infty)$-domains with fixed $\varepsilon$ is stable under Hausdorff convergence. An additional and related result is the Mosco convergence of Robin-type energy functionals on converging domains.  
\end{abstract}


\section{Introduction}

The first step towards the solution of a shape optimization problem for a given functional is to prove the existence of a shape which is optimal in a certain class of shapes in the sense that it minimizes the functional. In the context of a boundary value problem for a partial differential equation, the functional typically is an energy of the respective solution, and the class of shapes in which an optimal one sought for is a class of domains. Examples for such classes of domains are the collections of all Lipschitz domains contained in a given bounded open set $D$ and satisfying the $\varepsilon$-cone condition for the same $\varepsilon>0$, see \cite[Section 2.4]{HENROT-e}. One specific feature of these classes of domains is their compactness with respect to the convergence in the Hausdorff sense, in the sense of compacts, and in the sense of characteristic functions, \cite[Theorem 2.4.10]{HENROT-e}. This is significant because suitable compactness properties are a prerequisite needed to prove the existence of an optimal shape, \cite{CHENAIS-1975,FEIREISL-2002-1,HENROT-e}. A second specific feature of these classes is that their elements $\Omega$ are Sobolev extension domains, \cite{CALDERON-1961, CHENAIS-1975, JONES-1981, STEIN-1970}, and moreover, that the linear extension operators extending an element of $W^{k,p}(\Omega)$ to an element of $W^{k,p}(\R^n)$ have a norm bound depending only on $n$, $\varepsilon$, $p$ and $k$, and therefore valid uniformly for all domains $\Omega$ in the fixed class. A third specific feature, particularly useful to discuss boundary value problems in variational formulation, is that for domains $\Omega$ from these classes, there are bounded linear trace and extension operators between $W^{1,2}(\Omega)$ and suitable function spaces on the boundary $\partial\Omega$ and that their operator norms, too, are uniformly bounded for all domains in the class.

We are motivated by recent results on the existence of optimal shapes, realizing the infimum of the acoustical energy for a frequency boundary absorption problem over a class of Lipschitz domains~\cite{optimal}. There the optimal shapes themselves were not necessarily elements of the same class so that the infimum cannot be claimed to be a minimum. We are also motivated by corresponding numerical experiments~\cite[Section 5]{optimalandefficient}, in which a multiscale behavior of the optimal shapes was observed. This multiscale behavior is needed to have almost optimal shapes on a fixed bounded range of frequencies~\cite[Section~5.2]{optimalandefficient}. As the number of geometrical scales grows with the considered frequency range (with
sizes converging to $0$), this leads to fractal type shapes~\cite{optimalPhys} on an unbounded frequency range.

Here we address shape optimization problems for certain classes of domains more general than Lipschitz domains. Well-established results on extension operators \cite{JONES-1981, Rogers}, and classical results on their geometric structure \cite{MARTIO-1979, Vaisala88}, suggest to look at classes of bounded $(\varepsilon,\infty)$-domains, also referred to as \emph{bounded uniform domains}, for fixed $\varepsilon>0$. These classes contain Lipschitz domains, but also domains with rough non-Lipschitz boundary, such as snowflake domains \cite{WALLIN-1991}. Since we are particularly interested in the existence of optimal shapes for certain mixed boundary value problems, also trace and extension results for the boundaries of such domains matter. Classical snowflake domains have boundaries that are $d$-sets, for which properties of trace and extension operators are relatively well-known \cite{JONSSON-1984, TRIEBEL-1997}. 
By definition $d$-sets are, roughly speaking, everywhere of a fixed Hausdorff dimension $n-1\leq d\leq n$, see Remark~\ref{R:compare} below. Using more general trace results as in \cite[Chapter 7]{AH96} and \cite{Biegert2009} or \cite{JONSSON-1994}, we can also permit (bounded) $(\varepsilon,\infty)$-domains whose boundaries may have parts of different Hausdorff dimensions. This seems particularly adequate for shape optimization problems in which parts of the boundary may vary, but other (and possibly more regular) parts are kept fixed. 
To satisfy the respective hypothesis of these trace results, the boundaries have to carry measures satisfying specific scaling properties, see Section~\ref{S:measures}. To discuss the existence of optimal shapes in a certain class of domains, we have to discuss the convergence of measures on the boundaries and to guarantee the stability of the specified class of domains under this convergence. This can be done because the mentioned scaling conditions behave well under weak convergence of measures, as observed in Proposition~\ref{prop-A} and Lemma~\ref{measure-lemma}. This fact is in line with the observation made in~\cite{optimal} that weak limits of $(n-1)$-dimensional Hausdorff measures on Lipschitz boundaries may not exactly be Hausdorff measures again, but measures equivalent to Hausdorff measures.

We implement the mentioned ideas in Definition~\ref{DefShapeAdmis}, where we define parametrized classes of bounded $(\varepsilon,\infty)$-domains $\Omega$ in $\mathbb{R}^n$, $n\geq 2$, with fixed $\varepsilon>0$, that are all contained in a fixed bounded open set $D\subset \mathbb{R}^n$ and all contain a fixed non-empty open set $D_0$ (to prevent them from collapsing to the empty set under Hausdorff convergence) and whose boundaries $\partial\Omega$ are the supports of Borel measures satisfying Ahlfors regularity type conditions with fixed exponents and constants, we call them \emph{boundary volumes}. Since these parametrized classes of domains are well suited to shape optimization problems, we refer to their elements as \emph{shape admissible domains}. The boundaries of shape admissible domains may have pieces that are smooth or Lipschitz, self-similar fractals, or $d$-sets, and in general, they may have multifractal structure. Our main result on shape admissible domains is Theorem~\ref{T:compact}. It asserts that each parametrized class of shape admissible domains is compact with respect to the convergence in the Hausdorff sense, in the sense of compacts, in the sense of characteristic functions, and the sense of weak convergence of the boundary volumes. It also concludes that the weak convergence of the boundary measures entails the other convergences just mentioned. One ingredient for this result is Theorem~\ref{T:epsinftystable}, which may be of independent interest. It states that the class of bounded $(\varepsilon,\infty)$-domains with fixed $\varepsilon>0$, and contained in a fixed bounded open set $D\subset \mathbb{R}^n$ (with $n\geq 2$), is stable under convergence in the Hausdorff sense. To connect these results with applications to partial differential equations, we review the trace and extension results from \cite{AH96, Biegert2009, JONSSON-1994} and from \cite{JONES-1981, Rogers} and collect some consequences in Theorem~\ref{ThGENERIC} and subsequent corollaries. To extend functions from closed sets to the whole space, we can use harmonic extension operators, which in the case of an $L^2$-framework are linear. These trace and extension results are formulated for a wider class of domains, which (roughly following \cite{ARFI-2017, DEKKERS-2020, ROZANOVA-PIERRAT-2020}) we call \emph{$W^{1,2}$-admissible domains}, Definition~\ref{DefSAdmis}. As the first application to boundary value problems, we show the Mosco convergence of energy functionals associated with Robin problems on a sequence of suitably convergent domains, Theorem~\ref{T:Mosco}. We then turn to linear acoustics and state the well-posedness of a mixed boundary value problem for the Helmholtz equation on $W^{1,2}$-admissible domains, Theorem~\ref{ThWPH}. This is a generalization of an analogous result in~\cite{optimal} for $d$-set boundaries. 
On a fixed class of shape admissible domains, the (weak) solutions of the Helmholtz problems admit a uniform bound. This allows
to conclude the existence of an optimal shape minimizing the acoustical energy by absorption, Theorem~\ref{ThExistOptimalShape}. In \cite{optimal}, an analogous existence result had been proved in the framework of Lipschitz boundaries.
 
Related results on linear wave propagation problems with an irregular boundary can be found in~\cite{C-WHM2017,C-WH2018,CHM2019}.
Shape optimization problems in the context of fluid dynamics had been solved in~\cite{FEIREISL-2002-1} for domains with uniform thick boundaries, a class suitable to study problems with homogeneous Dirichlet boundary conditions. Further results on compact classes of admissible domains developed for problems with homogeneous Dirichlet boundary condition can be found in \cite{SVERAK-1993} (dimension two) and \cite{BUCUR-1995} (higher dimensions). A free discontinuity approach to a class of shape optimization problems involving a Robin condition on a free boundary had been studied in~\cite{BUCUR-2016}. Very recent results on shape optimization in classes of uniform domains, but with goals and methods quite different from ours, can be found in \cite{DuLiWang}.

The paper is organized as follows. 
In~Section~\ref{sec-epsilon-infty} we prove the stability of bounded $(\varepsilon,\infty)$-domains under Hausdorff convergence,~Theorem~\ref{T:epsinftystable}.
In~Section~\ref{S:measures}, we recall the scaling properties for measures, including those specified in \cite{JONSSON-1994}, and verify their stability under weak convergence. 
In~Section~\ref{S:domains}, we recall several different notions of convergence for domains and prove compactness results, Theorem~\ref{T:nothing_extra}. 
We then define parametrized classes of shape admissible domains, Definition~\ref{DefShapeAdmis}, and finally prove   Theorem~\ref{T:compact} on compactness and convergence. 
In~Section~\ref{S:traceExt}, we introduce the functional framework needed for the well-posedness of various problems, including the Helmholtz problem and for the shape optimization problem and collect properties of the trace and extension operators acting on $W^{1,2}$-admissible domains, Theorem~\ref{ThGENERIC}. 
The Mosco convergence result, Theorem~\ref{T:Mosco}, is proved in ~Section~\ref{S:Mosco}. 
In~Section~\ref{S:shapeOp} discuss the well-posedness of a mixed boundary valued problem for the Helmholtz equation, Theorem~\ref{ThWPH}, and solve the existence problem for an optimal shape in a fixed class of shape admissible domains in Theorem~\ref{ThExistOptimalShape}.


%

By $B(x,r)$ we denote the open ball in $\R^n$ centered at $x$ and of radius~$r$. We write $\lambda^n$ for the $n$-dimensional Lebesgue measure.
We assume $n\geq 2$ throughout the paper. 

\section{Bounded uniform domains and their Hausdorff convergence stability}\label{sec-epsilon-infty}

For the class of bounded Lipschitz domains satisfying the cone condition with the same parameter 
the stability  under Hausdorff convergence is proved in 
\cite[Theorem 2.4.10]{HENROT-e}. 
We aim at a larger class, based on the following classical definition, \cite{AHMNT, JONES-1981,JONSSON-1984, MARTIO-1979, Vaisala88, WALLIN-1991}.  Recall that a domain in $\mathbb{R}^n$ is a connected open subset of $\mathbb{R}^n$. 

\begin{definition}\label{DefEDD}
	Let $\eps > 0$. A bounded domain $\Omega\subset \R^n$ is called an \emph{$(\eps,\infty)$-uniform domain} if for all $x, y \in \Omega$ there is a rectifiable curve $\gamma\subset \Omega$ with length $\ell(\gamma)$ joining $x$ to $y$ and satisfying
	\begin{enumerate}
		\item[(i)] $\ell(\gamma)\le \frac{|x-y|}{\eps}$ and
		\item[(ii)] $d(z,\del \Omega)\ge \eps |x-z|\frac{|y-z|}{|x-y|}$ for $z\in \gamma$.
	\end{enumerate}
\end{definition}

\begin{remark}\label{R:cigars}
	Condition (ii) is equivalent to saying that, in the terminology of \cite[2.1 and 2.4]{Vaisala88}, the \emph{$\frac{1}{\varepsilon}$-cigar} 
	\begin{equation}\label{E:cigar}
		C(\gamma,\varepsilon):=\bigcup_{z\in\gamma} B(z,\varepsilon\lambda(z)),\quad  \text{where}\quad \lambda(z)=|x-z|\frac{|y-z|}{|x-y|},\quad z\in\gamma,
	\end{equation}
	is contained in $\Omega$. See also \cite[2.1 and 2.12]{MARTIO-1979}.
\end{remark}


For any closed set $K\subset \mathbb{R}^n$ and $\alpha>0$ we write $(K)_\alpha:=\{x\in\mathbb{R}^n: d(x,K)\leq \alpha\}$ for its closed (outer) $\alpha$-parallel set. Recall that the \emph{Hausdorff distance} between two compact sets $K_1, K_2\subset \mathbb{R}^n$ is defined as 
\[d^H(K_1,K_2):=\inf\{\alpha>0: K_1\subset (K_2)_\alpha \text{ and } K_2\subset (K_1)_\alpha\}.\]  
A sequence $(K_m)_m$ of compact sets $K_m\subset \mathbb{R}^n$ is said to \emph{converge} to a compact set $K\subset \mathbb{R}^n$ \emph{in the Hausdorff sense} if $\lim_{m\to \infty} d^H(K_m,K)=0$.

Let $D\subset \mathbb{R}^n$ be a bounded open set. A sequence $(\Omega_m)_m$ of open sets $\Omega_m\subset D$ is said to \emph{converge} to an open set $\Omega\subset D$ \emph{in the Hausdorff sense} if 
\[d^H(\overline{D}\setminus \Omega_m,\overline{D}\setminus \Omega)\to 0\quad \text{ as }\quad m\to \infty,\] 
\cite[Definition 2.2.8]{HENROT-e}, which does not depend on the choice of $D$, \cite[Remark 2.2.11]{HENROT-e}.

\begin{theorem}\label{T:epsinftystable}
	Let $D\subset \mathbb{R}^n$ be a bounded open set and let $\varepsilon>0$. Any sequence $(\Omega_m)_m$ of 
	$(\varepsilon,\infty)$-domains 
	contained in $D$  has a subsequence which converges to an $(\varepsilon,\infty)$-domain $\Omega\subset D$ in the Hausdorff sense. 
\end{theorem}

This implies in particular that the limit $\Omega\subset D$ in the Hausdorff sense of a convergent sequence $(\Omega_m)_m$ of $(\varepsilon,\infty)$-domains $\Omega_m\in D$ is an $(\varepsilon,\infty)$-domain.

Our proof of  Theorem~\ref{T:epsinftystable} is based on  Remark~\ref{R:cigars}. Recall that the \emph{Fr\'echet distance} between two curves $\gamma_1,\gamma_2\subset \mathbb{R}^n$ is defined as
\[d^F(\gamma_1,\gamma_2):=\inf_{(g_1,g_2)}\max_{t\in [0,1]} d(g_1(t),g_2(t)),\]
where the infimum is taken over all pairs $(g_1,g_2)$ of parametrizations $g_i:[0,1]\to\mathbb{R}^n$ of $\gamma_i$, $i=1,2$, \cite[Section I.1.4]{ALEXANDROV-RESHETNYAK}. A sequence $(\gamma_m)_m$ of curves $\gamma_m\subset \mathbb{R}^n$ is said to \emph{converge} to a curve $\gamma\subset \mathbb{R}^n$ \emph{in the Fr\'echet sense} if $d^F(\gamma_m,\gamma)\to 0$ as $m\to\infty$.

\begin{lemma}\label{L:cigar}
	Let $D\subset \mathbb{R}^n$, $n\geq 2$, be a bounded open set and $\varepsilon>0$. Suppose that $(\gamma_m)_m$ is a sequence of rectifiable curves $\gamma_m$ with distinct end points $x_m$ and $y_m$ in $D$, respectively, such that $\ell(\gamma_m)\leq \frac{|x_m-y_m|}{\varepsilon}$ and $C(\gamma_m,\varepsilon)\subset D$ for all $m$. If $x$ and $y$ are distinct points in $D$ such that $x_m\to x$ and $y_m\to y$, then there are a sequence of indexes $(m_k)_k$ and a rectifiable curve $\gamma$ connecting $x$ and $y$ with $\ell(\gamma)\leq \frac{|x-y|}{\varepsilon}$ and $C(\gamma,\varepsilon)\subset D$.
\end{lemma}

\begin{proof}
	Since $\overline{D}$ is compact, we can find a sequence of indexes $(m_k)_k$ and a rectifiable curve $\gamma$ of length $\ell(\gamma)\leq \lim_k \ell(\gamma_{m_k})$ such that $\gamma_{m_k}\to \gamma$ in $\overline{D}$ in the Fr\'echet sense as $k\to \infty$, \cite[Theorems 2.1.5 and 2.1.2]{ALEXANDROV-RESHETNYAK}. To save notation, we relabel and denote this convergent sequence again by $(\gamma_m)_m$. One can find suitable parametrizations $g:[0,1]\to \overline{D}$ and $g_m:[0,1]\to \overline{D}$ for $\gamma$ and $\gamma_m$ such that $\lim_m g_m=g$ uniformly on $[0,1]$, \cite[Lemma 1.4.1]{ALEXANDROV-RESHETNYAK}. This implies in particular that, without loss of generality, $\gamma(0)=x$ and $\gamma(1)=y$. Given $\alpha>0$, consider the open inner $\alpha$-parallel set 
	\begin{equation}\label{E:psetnew}
		\left\lbrace y\in C(\gamma,\varepsilon):d(y,\partial C(\gamma,\varepsilon))>\alpha\right\rbrace =\bigcup_{t\in [0,1]: \varepsilon\lambda(g(t))>\alpha} B(g(t),\varepsilon\lambda(g(t))-\alpha)
	\end{equation}
	of $C(\gamma,\varepsilon)$. For any sufficiently large $m$ we have 
	\[\left||x_m-y_m|^{-1}-|x-y|^{-1}\right|<\alpha(2\varepsilon)^{-1}(\diam(D))^{-2}\wedge (\alpha/2)\]
	and 
	\[\sup_{t\in [0,1]}|g(t)-g_m(t)|<\alpha|x-y|(8\varepsilon \diam(D))^{-1}\wedge (\alpha/2).\]
	Writing $\lambda_m$ for the function defined as $\lambda$ in ~\eqref{E:cigar}, but with $x_m$, $y_m$ and $\gamma_m$ in place of $x$, $y$ and $\gamma$, respectively, we observe that for any such $m$ we have 
	\begin{align}
		\varepsilon|\lambda(g(t))&-\lambda_m(g_m(t))|\notag\\
		&<\varepsilon \left||x_m-y_m|^{-1}-|x-y|^{-1}\right| |x_m-g_m(t)||y_m-g_m(t)|\notag\\
		&\qquad\qquad +\varepsilon |x-y|\left||x_m-g_m(t)||y_m-g_m(t)|-|x-g(t)||y-g(t)|\right|\notag <\alpha\notag
	\end{align}
	for any $t\in [0,1]$. Consequently $B(g(t),\varepsilon \lambda(g(t))-\alpha)\subset B(g_m(t), \varepsilon \lambda_m(g_m(t))$ for all $t$ with $\lambda(g(t))>\alpha$, hence the set in ~\eqref{E:psetnew} is contained in 
	\[C(\gamma_m,\varepsilon)=\bigcup_{t\in [0,1]}B(g_m(t),\varepsilon\lambda_m(g_m(t)).\] 
	In a similar fashion we see that for such $m$ the set $\{y\in C(\gamma_m,\varepsilon): d(y,\partial C(\gamma_m,\varepsilon))>\alpha\}$ is contained in $C(\gamma,\varepsilon)$. Together this shows that $\overline{D}\setminus C(\gamma_m,\varepsilon)\subset (\overline{D}\setminus C(\gamma,\varepsilon))_\alpha$ and $\overline{D}\setminus C(\gamma,\varepsilon)\subset (\overline{D}\setminus C(\gamma_m,\varepsilon))_\alpha$ for large $m$, what shows that $C(\gamma_m,\varepsilon)\to C(\gamma,\varepsilon)$ in Hausdorff sense as $m\to \infty$.
\end{proof}

\begin{proof}[Proof of Theorem~\ref{T:epsinftystable}]
	If $(\Omega_m)_m$ is a sequence of $(\varepsilon,\infty)$-domains contained in $D$, then by  \cite[Corollary 2.2.26]{HENROT-e}
	we can find a sequence of indexes $(m_k)_k$ and an open set $\Omega\subset D$ so that $\Omega_{m_k}\to \Omega$ in the Hausdorff sense as $k\to \infty$. To save notation, we relabel and denote this sequence again by $(\Omega_m)_m$. If $x$ and $y$ are two different points in $\Omega$, then by \cite[Proposition 2.2.17]{HENROT-e} both $x$ and $y$ belong to $\Omega_{m}$ for any large enough $m$. 
	For each such $m$ let $\gamma_m\subset \Omega_m$ be a rectifiable curve of length $\ell(\gamma_m)\leq \frac{|x-y|}{\varepsilon}$ connecting $x$ and $y$ and let $C(\gamma_m,\varepsilon)$ be as in ~\eqref{E:cigar}, but with $\gamma_m$ in place of $\gamma$. An application of Lemma~\ref{L:cigar} with $x_m=x$ and $y_m=y$ for all $m$ shows the existence of a sequence of indexes $(m_k)_k$ and a rectifiable curve $\gamma$ connecting $x$ and $y$ such that the sets $C(\gamma_{m_k},\varepsilon)$ converge to $C(\gamma,\varepsilon)$ in the Hausdorff sense. Since $C(\gamma_m,\varepsilon)\subset \Omega_m$ for each $m$, it follows that $C(\gamma,\varepsilon)\subset \Omega$, \cite[(2.16) in 2.2.3.2]{HENROT-e}.
\end{proof}

\section{Measures on closed subsets of $\mathbb{R}^n$, scaling and   stability}
\label{S:measures}

We consider Borel measures on closed subsets of $\mathbb{R}^n$ having specific scaling properties. 
In later sections, we will study boundary value problems, when the closed subsets under consideration can be the boundaries of the respective domains, and the measures can replace the surface measure.

\subsection{Stability of lower and upper Ahlfors scaling conditions} 
For a Borel measure $\mu$ with $K:=\supp \mu$, exponents $0<s\leq n$, $0\leq d\leq n$ and constants $c_s^A>0$ and $c_d^A>0$
we  recall the local lower and upper Ahlfors regularity condition 
\begin{equation}\label{E:lowreg}
\mu(B(x,r))\geq c_{s}^A\:r^s, \quad x\in K,\quad  0<r\leq 1,
\end{equation}
which implies $\dim_H K \leq s$, where $\dim_H K$ denotes the Hausdorff dimension of $K$, \cite{FALCONER,MATTILA}, and 
\begin{equation}\label{E:upreg}
\mu(B(x,r))\leq c_{d}^A\:r^d, \quad x\in K, \quad 0<r\leq 1,
\end{equation}
which implies $\dim_H K \geq d$. If $\mu$ satisfies both ~\eqref{E:lowreg} and ~\eqref{E:upreg}, then $d\leq s$.

\begin{remark}\label{E:potentialsarefine}
Obviously, any Borel measure satisfying ~\eqref{E:upreg} is locally finite. Note also that any Borel measure on $\mathbb R^n$ is regular by \cite[Theorem 2.18]{Rudin}.
\end{remark}

Because it has convenient stability properties (see Proposition~\ref{prop-A} below),
we also introduce the local lower Ahlfors regularity condition
\begin{equation}\label{E:lowreg-w}
\mu(\overline{B(x,r)})\geq \bar c_{s}^A\:r^s, \quad x\in  K,\quad  0<r\leq 1.
\end{equation}
Varying the radii, \eqref{E:lowreg-w} is easily seen to be equivalent to \eqref{E:lowreg}, and using \cite[Theorem 6.9 (ii)]{MATTILA} we therefore obtain the following.	
	
	\begin{proposition}\label{C:Kisnullset}
		If $0 < s<n$  in   ~\eqref{E:lowreg-w},   then  $K$ has empty interior,
		$\lambda^n(K)=0$, and $\dim_H K\leq s$.
	\end{proposition}

As usual we say that a sequence $(\mu_m)_m$ of Borel measures $\mu_m$ 
\emph{converges weakly} 
to a Borel measure $\mu$ if  
\[\lim_{m\to \infty} \int_{\mathbb{R}^n} f\:d\mu_m=\int_{\mathbb{R}^n} f\:d\mu,\quad f\in C_b(\mathbb{R}^n).\]

The following convenient stability result is relatively well known.

\begin{proposition}\label{prop-A}
Suppose that $D\subset \mathbb{R}^n$ is a bounded open set, $\mu_m$ are Borel measure with $\supp\mu_m\subset \overline{D}$ for all $m$ and that $\mu_m \to \mu $  weakly. 
	\begin{enumerate}  
		\item[(i)] If $\mu_m$ satisfy ~\eqref{E:upreg}, then the limit measure $\mu$ also satisfies ~\eqref{E:upreg}. 
		\item[(ii)] If $\mu_m$ satisfy   ~\eqref{E:lowreg-w},   then the limit  $\mu$ also satisfies  
 ~\eqref{E:lowreg-w}
  and $$\supp\mu_m \to \supp \mu \quad \hbox{for } m\to +\infty$$  in the Hausdorff sense. 
	\end{enumerate}
\end{proposition}


\begin{proof} Denote $K:=\supp\mu$ and $K_m:=\supp\mu_m$. 
	By \cite[Theorem 2.2.25]{HENROT-e}, we can find a subsequence of $(K_m)_m$ with a limit $K'$ in the Hausdorff sense. 
	As can be seen from the proof, the subsequence choice does not matter, so we write again $(K_m)_m$ for this subsequence and $\mu_m$ for the measure with support $K_m$. By weak convergence, it is clear that $K\subset K'$. 
	Let $x\in K'$. Then there is a sequence $(x_m)_m$ of points $x_m\in K_m$ such that $\lim_m x_m=x$, see \cite[Proposition 2.2.27]{HENROT-e}. Given $0<\delta<r$, we have $B(x_m,r-\delta)\subset B(x,r)\subset B(x_m,r+\delta)$, and with the Portmanteau theorem it follows that $\mu(B(x,r))\leq \varliminf_m \mu_m(B(x,r))\leq \varlimsup_m \mu_m(B(x_m, r+\delta))$ and $\mu(\overline{B(x,r)})\geq \varlimsup_m \mu_m(\overline{B(x,r)})\geq \varliminf_m \mu_m(\overline{B(x_m,r-\delta)})$. The existence of some $x\in K'\setminus K$ would contradict ~\eqref{E:lowreg-w} and the last conclusion, thus $K'=K$. 	
\end{proof}

\subsection{Refined scaling conditions and their stability}\label{SS:refined} In \cite{JONSSON-1994} more refined scaling properties were key assumptions for trace and extension results for Besov spaces on   closed subsets of  $\mathbb{R}^n$. 
The conditions and results in \cite{JONSSON-1994} allow to treat measures  having non-integer Hausdorff dimensions,\cite{MattilaMoranRey}, and consisting of various parts having different Hausdorff dimensions. 
Global versions of  ~\eqref{EqMuDs} were studied in detail in \cite{VolbergKonyagin}, motivated by  \cite{Assouad80, Dyn'kin84}. Condition ~\eqref{EqMuLd} seems to have been introduced by Jonsson for the first time; see also \cite{BylundGudayol, LuukkainenSaksman}.

A Borel measure $\mu$ on $\mathbb{R}^n$ with support $K:=\supp \mu$ \emph{satisfies the $D_s$-condition} for an exponent $0<s\le n$ if there is a constant $c_s>0$ such that 
\begin{equation}\label{EqMuDs}
\mu(B(x,kr))\le c_s k^s \mu(B(x,r)), \quad x\in K, \quad r>0,\quad k\ge 1, \quad 0<k r\le 1.
\end{equation}
It is said to \emph{satisfy the $L_d$-condition} for an exponent $0\leq d\le n$ if for some constant $c_d>0$ we have 
\begin{equation}\label{EqMuLd}
\mu(B(x,kr))\ge c_d k^d \mu(B(x,r)), \quad x\in K, \quad r>0,\quad k\ge 1, \quad 0<k r\le 1.
\end{equation}
Apart from ~\eqref{EqMuDs} and ~\eqref{EqMuLd} we will also consider the condition
\begin{equation}\label{Eqnormalized}
c_1\leq \mu(B(x,1))\leq c_2,\quad x\in K,
\end{equation}
where $c_1>0$ and $c_2>0$ are constants independent of $x$.

\begin{remark}\label{R:compare} 
Combining~~\eqref{EqMuDs} and~~\eqref{Eqnormalized} one can find a constant $c_{s}^A>0$ such that 
 ~\eqref{E:lowreg} holds. Similarly  ~\eqref{EqMuLd} and ~\eqref{Eqnormalized} yield a constant $c_{d}^A>0$ such that 
 ~\eqref{E:upreg} holds. Moreover, ~\eqref{EqMuDs} implies the     doubling condition $\mu(B(x,2r))\leq c\:\mu(B(x,r))$, $x\in K$, $0<r\leq 1/2$, where $c>0$ is a suitable constant, \cite[Section 1]{JONSSON-1994}. If a Borel measure $\mu$ with support $K$ satisfies ~\eqref{E:lowreg} and ~\eqref{E:upreg} with $s=d$ for some $0<d\leq n$, then $\mu$ is called a \emph{$d$-measure} and $K$ is called a \emph{$d$-set}, see for instance \cite{JONSSON-1984,
JONSSON-1995,WALLIN-1991,TRIEBEL-1997}. The boundary of a Lipschitz domain, endowed with the $(n-1)$-dimensional Hausdorff measure $\mathcal{H}^{n-1}$, is an $(n-1)$-set. 
\end{remark}

\begin{remark}\label{R:triviallowerbound}\mbox{}
\begin{enumerate}
\item[(i)] If the closed set $K$ is the union of two closed sets $K_1$ and $K_2$ supporting measures $\mu_1$ and $\mu_2$ that meet conditions ~\eqref{EqMuDs}, ~\eqref{EqMuLd} and ~\eqref{Eqnormalized}, possibly with different constants and exponents, then the measure $\mu=\mu_1+\mu_2$ on $K$ satisfies ~\eqref{EqMuDs}, ~\eqref{EqMuLd} and ~\eqref{Eqnormalized} 
with readjusted constants, see \cite[Example 2]{JONSSON-1994}. 
\item[(ii)] If $\mu$ is a Borel measure on $\mathbb{R}^n$ whose support $K=\supp \mu$ is compact, then we can always find a constant 
$c_1>0$ so that the lower bound in ~\eqref{Eqnormalized} is satisfied: By Fatou's lemma the function $x\mapsto \mu(B(x,1))$ is lower semicontinous on $K$, hence $c_1:=\min_{x\in K} \mu(B(x,1))$ is attained at some $x_0\in K$. But this minimum must be strictly positive, otherwise $\mu(B(x_0,1))=0$, a contradiction.  
\end{enumerate}
\end{remark}


For the next lemma we need to introduce estimates \begin{equation}\label{Eqmeasurebounds}
\overline{c}_1\leq \mu(\overline{B(x,1)})  \quad \text{ and } \quad   \mu(B(x,1))\leq c_2, \quad x\in K. 
\end{equation}

\begin{remark}\label{R:modlowerbound}
It is easy to see that ~\eqref{EqMuDs} implies $\mu(\overline{B(x,r)})\leq c_s  \mu({B(x,r)})$,  $x\in K$, $r>0$. 
Therefore, if a Borel measure $\mu$ satisfies ~\eqref{EqMuDs} and the first inequality in ~\eqref{Eqmeasurebounds}, then it satisfies the lower bound in ~\eqref{Eqnormalized} with $\overline{c}_1 /c_s$ in place of $c_1$.
\end{remark}

 \begin{proposition}\label{measure-lemma} Suppose that $D\subset \mathbb{R}^n$ is a bounded open set and $\mu_m$ are Borel measure with $\supp\mu_m\subset \overline{D}$
and satisfying 
 ~\eqref{EqMuDs},
~\eqref{EqMuLd} 
  and
 ~\eqref{Eqmeasurebounds}.
If $\mu_m \to \mu$ weakly, then $\supp \mu_m\to \supp\mu$ in the Hausdorff sense, and $\mu$ satisfies  ~\eqref{EqMuDs},
~\eqref{EqMuLd} 
  and
 ~\eqref{Eqmeasurebounds}.
\end{proposition}

\begin{proof} Since by Remark~\ref{R:compare} we have ~\eqref{E:lowreg}, the convergence of supports follows from Proposition~\ref{prop-A} (ii). By the Portmanteau theorem 
	\begin{equation}\label{E:Portmanteau}
	\sup_{\delta\in(0,r)}\varlimsup\limits_{ m\to\infty}\mu_m (B(x,r-\delta))\leq \mu(B(x,r)) \leq \varliminf\limits_{ m\to\infty}\mu_m (B(x,r))
	\end{equation}
	holds for any $x\in\mathbb{R}^n$ and $0<\delta<r$. If  $x\in \supp \mu$, then there is a sequence $(x_m)_m$ of points $x_m\in \supp \mu_m$ such that $\lim_m x_m=x$. Given $0<\delta<r$ we have $B(x,kr)\subset B(x_m,k(r+\delta))$ for any sufficiently large $m$ and therefore, using  ~\eqref{E:Portmanteau} and applying ~\eqref{EqMuDs} with $k(r+\delta)=(k(r+\delta)/(r-\delta))(r-\delta)$,
\begin{multline}
		\mu(B(x,kr))\le\varliminf\limits_{m\to\infty}\mu_m(B(x_m,k(r+\delta)))\notag\\
		\le c_s \left(\frac{k(r+\delta)}{r-\delta}\right)^s\varliminf\limits_{m\to\infty}\mu_m(B(x_m, r-\delta )) \le c_s \left(\frac{k(r+\delta)}{r-\delta}\right)^s \mu(B(x,r)),
		\end{multline}
which proves ~\eqref{EqMuDs},  because $\delta$ can be arbitrarily small.  Estimate ~\eqref{EqMuLd}   follows similarly using ~\eqref{E:Portmanteau} and $k(r-\delta)=(k(r-\delta)/(r+\delta))(r+\delta)$, note that 
		\begin{multline}\label{EqMuLd-}
		\mu(B(x,kr))\ge\varliminf\limits_{m\to\infty}\mu_m(B(x_m,k(r-\delta)))\notag\\
		\ge c_d \left(\frac{k(r-\delta)}{r+\delta}\right)^d\varliminf\limits_{m\to\infty}\mu_m(B(x_m, r+\delta )) \ge c_d \left(\frac{k(r-\delta)}{r+\delta}\right)^d \mu(B(x,r)).
\end{multline}	
The estimates ~\eqref{Eqmeasurebounds} and follows as in Proposition~\ref{prop-A}. 
\end{proof}

\section{Convergence of sequences of domains, stability, and compactness}\label{S:domains}

We consider the convergence properties of certain classes of bounded domains. A first compactness result is~Theorem~\ref{T:nothing_extra} (i), which concludes the existence of subsequential limits in the sense of weak convergence of measures on the boundary, in the Hausdorff sense and in the characteristic function sense for domains confined to a bounded open set. Similarly, as in \cite{HENROT-e} the assumption of convergence in the sense of compacts can prevent the limit measure from having support larger than the boundary of the (subsequential) limit domain, Theorem~\ref{T:nothing_extra} (ii). 
It allows a refined compactness result, including convergence in the sense of compacts,~Theorem~\ref{T:compact}, for certain classes of domains with fixed quantitative specifications,~Definition~\ref{DefShapeAdmis}.

\subsection{Sequences of domains and boundaries}


 A sequence $(\Omega_m)_m$ of open sets $\Omega\subset \mathbb{R}^n$ is said to \emph{converge} to an open $\Omega$ \emph{in the sense of characteristic functions} if 
\[\lim_{m\to \infty} \mathds{1}_{\Omega_m}=\mathds{1}_{\Omega}\quad \text{in $L^p_{\loc}(\mathbb{R}^n)$ for all $p\in [1,\infty)$},\]
\cite[Definition 2.2.3]{HENROT-e}. A sequence $(\Omega_m)_m$ of open sets $\Omega_m\subset \mathbb{R}^n$ is said to \emph{converge} to an open set $\Omega\subset \mathbb{R}^n$ \emph{in the sense of compacts} if for any compact $L\subset \Omega$ we have $L\subset \Omega_m$ for all sufficiently large $m$ and for any compact $L'\subset \mathbb{R}^n\setminus \overline{\Omega}$ we have $L'\subset \mathbb{R}^n\setminus \overline{\Omega}_m$ for all sufficiently large $m$.

	To a finite Borel measure $\mu$ whose support is the boundary $\partial\Omega$ of a bounded domain $\Omega\subset \mathbb{R}^n$, $\supp\mu=\partial\Omega$, we refer as a \emph{boundary volume} for $\Omega$. The topological dimension of $\partial\Omega$ is $n-1$. (All well-established concepts of topological dimension agree for separable metric spaces, see \cite[Chapter 3, p. 110]{Edgar2008} and the references cited there or \cite[Chapter V]{HurewiczWallman1941}.) Since the Hausdorff dimension of a set is greater or equal to its topological dimension, see \cite[Theorem 6.3.10]{Edgar2008} or \cite[Theorem VII.2]{HurewiczWallman1941}, we therefore have $n-1\leq \dim_H \partial\Omega\leq n$. It follows that if a boundary volume satisfies (\ref{E:lowreg}), then the exponent $s$ in (\ref{E:lowreg}) must satisfy $n-1\leq  s\leq n$.
 
\begin{theorem}\label{T:nothing_extra}
Let $D\subset \mathbb{R}^n$ be bounded and open and $n-1\leq  s<n$. Suppose that $(\Omega_m)_m$ is a sequence of domains $\Omega_m\subset D$ and $(\mu_m)_m$ is a sequence of boundary volumes $\mu_m$ for the domains $\Omega_m$, respectively, which satisfy 
 ~\eqref{E:lowreg-w}
with $\supp\mu_m=\partial \Omega_m$. 
\begin{enumerate}
\item[(i)] If $\sup_m\mu_m(\overline{D})<+\infty$, then there are a sequence $(m_k)_k$ of indexes and an open set $\Omega\subset D$ such that the sequence $(\mu_{m_k})_k$ converges weakly   to a  Borel measure $\mu$   satisfying   ~\eqref{E:lowreg-w} with $K=\supp\mu$, and we have $\partial \Omega\subset K\subset \overline{D}\setminus \Omega$. The sequences $(\Omega_{m_k})_k$ and $(D\setminus \overline{\Omega}_{m_k})_k$ converge to $\Omega$ and $D\setminus (\Omega\cup K)$, respectively, in   the Hausdorff sense and the sense of characteristic functions.
\item[(ii)] If $(\Omega_m)_m$ converges to $\Omega$ in the Hausdorff sense and in the sense of compacts and $(\mu_m)_m$ converges weakly to a Borel measure $\mu$, then $\supp \mu=\partial\Omega$ and $\mu$ satisfies   ~\eqref{E:lowreg-w}.
\end{enumerate}
\end{theorem}

\begin{proof}
We use the notation $d_K(x):=d(x,K)$ for closed $K\subset \mathbb{R}^n$. To see the subequential Hausdorff convergence of domains we follow \cite[Theorem 2.2.25 and Corollary 2.2.26]{HENROT-e}. 
The $f_m=(f_m^{(1)},f_m^{(2)}):\overline{D}\to \mathbb{R}_+\times \mathbb{R}_+$, defined by $f_m(x)=(d_{\overline{D}\setminus \Omega_m}(x), d_{\overline{\Omega}_m}(x))$, form an equibounded sequence $(f_m)_m$, and since 
\begin{equation}\label{E:AAequicont}
|d_{\overline{D}\setminus \Omega_m}(x)-d_{\overline{D}\setminus \Omega_m}(y)|\leq d(x,y)\quad \text{ and }\quad 
|d_{\overline{\Omega}_m}(x)-d_{\overline{\Omega}_m}(y)|\leq d(x,y)
\end{equation}
for all $x,y\in \overline{D}$, it is also equicontinous. By Arzela-Ascoli we can find a sequence $(m_k)_k$ so that $(f_{m_k})_k$ converges to a continuous function $f=(f^{(1)},f^{(2)}):\overline{D}\to \mathbb{R}_+\times \mathbb{R}_+$. We claim that writing $\Omega:=\{f^{(1)}>0\}$ and $\Omega':=\{f^{(2)}>0\}$ we have $f=(d_{\overline{D}\setminus \Omega}, d_{\overline{D}\setminus \Omega'})$. To see this, note first that $\overline{D}\setminus \Omega = \{f^{(1)}=0\}$ and taking limits in ~\eqref{E:AAequicont} yields $f^{(1)}\leq d_{\overline{D}\setminus \Omega}$. Given $x\in \overline{D}$ let $x_m\in \overline{D}\setminus \Omega_m$ be such that $d_{\overline{D}\setminus \Omega_m}(x)=d(x,x_m)$. By the compactness of $\overline{D}$ (and passing to further subsequences if needed) we can find a sequence $x_{m_k}$ with limit $y\in\overline{D}$ so that $f^{(1)}(x)=\lim_k f^{(1)}_{m_k}(x)=\lim_k d(x,x_{m_k})=d(x,y)$. Since $f^{(1)}(y)\leq \lim_k d(y,x_{m_k})=0$ we have $y\in \overline{D}\setminus \Omega$ and therefore, and since $x$ was arbitrary, $f^{(1)}\geq d_{\overline{D}\setminus \Omega}$. Similarly we can see that $f^{(2)}= d_{\overline{D}\setminus \Omega'}$. Now it follows from \cite[Proposition 2.2.27]{HENROT-e} that $\Omega_{m_k}\to \Omega$ and $D\setminus \overline{\Omega}_{m_k}\to \Omega'$ in the Hausdorff sense. Since $K:=\lim_{k\to \infty} \partial\Omega_{m_k}$ equals $\{f=0\}$ and $\{f^{(1)}>0, f^{(2)}>0\}=\emptyset$, it follows that $\Omega'=D\setminus (\Omega\cup K)$.

By the Banach-Alaoglu theorem we may, passing to a further subsequence, assume that the boundary volumes converge weakly to a  Borel measure $\mu$. By Proposition~\ref{prop-A} its support $K=\supp\mu$ is the limit in the Hausdorff sense of the boundaries $\partial\Omega_{m_k}$, and $\mu$ satisfies  ~\eqref{E:lowreg-w}. By \cite[Proposition 2.2.16, and the remarks following it]{HENROT-e} we have $\partial\Omega\subset K$. 
By definition the open sets $D\setminus \partial\Omega_{m_k}$ converge in the Hausdorff sense to $D\setminus K$, and clearly $\Omega_{m_k}\subset D\setminus \partial\Omega_{m_k}$ for all $k$. According to \cite[(2.16) in 2.2.3.2]{HENROT-e} it follows that $\Omega\subset  D\setminus K$, hence $K\subset \overline{D}\setminus \Omega$.

For simplicity we relabel and denote the chosen subsequence again by $(\Omega_m)_m$, and by Banach-Alaoglu we may assume that $(\mathds{1}_{\Omega_m})_m$ converges weakly$^\ast$ in $L^\infty(D)$ to a function $\chi$. Clearly $\chi\leq \mathds{1}_D$ $\lambda^n$-a.e. and \cite[Proposition 2.2.23]{HENROT-e} yields $\mathds{1}_\Omega\leq \chi$ $\lambda^n$-a.e. To show that also $\chi\leq \mathds{1}_\Omega$ $\lambda^n$-a.e. suppose that $\delta>0$. For sufficiently small $\alpha>0$ we have $\lambda^n((K)_\alpha)<\delta$ by Corollary~\ref{C:Kisnullset}, and in particular, $\int_{(\overline{\Omega})_\alpha\setminus \overline{\Omega}}\chi dx<\delta$. We claim that if $y\in D$ has distance $d(y,\overline{\Omega})>\alpha$ from $\overline{\Omega}$, then for any ball $B(y,r)$ with $r<\alpha/2$ we have 
\begin{equation}\label{E:claimintzero}
\int_{B(y,r)\cap D}\chi dx=0.
\end{equation}
If this is true, then covering the closure of $A_\alpha:=\{x\in\overline{D}: d(x,\overline{\Omega})>\alpha\}$ by finitely many such balls we can deduce that $\chi=0$ $\lambda^n$-a.e. on $A_\alpha$. Since $\lambda ^n(\partial\Omega)\leq \lambda ^n(K)=0$ by the preceding and Corollary~\ref{C:Kisnullset}, we obtain
\[\int_{\Omega^c}\chi dx=\int_{\partial\Omega}\chi dx+\int_{(\overline{\Omega})_\alpha\setminus \overline{\Omega}}\chi dx+\int_{A_\alpha}\chi dx<\delta,\]
and since $\delta$ was arbitrary, $\chi=0$ $\lambda^n$-a.e. on $\Omega^c$. This shows that $\chi=\mathds{1}_\Omega$, so that by \cite[Proposition 2.2.1]{HENROT-e} we have $\Omega_m\to \Omega$ in the sense of characteristic functions. To verify ~\eqref{E:claimintzero} let $y$ and $r$ be as there and suppose there exists some $\gamma>0$ such that $\int_{B(y,r)\cap D}\chi dx>2\gamma$. Choose $\beta>0$ so that $\lambda^n((K)_{2\beta})<\gamma$. By the weak$^\ast$ convergence in $L^\infty(D)$ we have 
\[\lim_m \lambda^n (B(y,r)\cap \Omega_m)=\int_D \mathds{1}_{B(y,r)}\mathds{1}_{\Omega_m} dx=\int_D\mathds{1}_{B(y,r)}\chi dx,\]
hence $ \lambda^n (B(y,r)\cap \Omega_m)>\gamma$ for all large enough $m$. Since $\overline{D}\setminus \Omega_m\to \overline{D}\setminus \Omega$ in the Hausdorff sense, it holds that $\overline{D\cap B(y,r)}\subset (\overline{D}\setminus \Omega_m)_\beta$ for all large enough $m$, hence $B(y,r)\cap \Omega_m\subset \{x\in \Omega_m: d(x,\overline{D}\setminus\Omega_m)\leq \beta\}$ and therefore, since $\partial\Omega_m\to K$, 
\[\lambda^n (B(y,r)\cap \Omega_m)\leq \lambda^n((\partial\Omega_m)_\beta)\leq \lambda^n((K)_{2\beta})<\gamma\]
for all large enough $m$, what contradicts the preceding. Consequently ~\eqref{E:claimintzero} holds. The convergence $D\setminus \overline{\Omega}_m \to D\setminus (\Omega\cup K)$ in the sense of characteristic functions is an immediate consequence, and this completes the proof of (i).

To see (ii) note that since by \cite[Proposition 2.2.16 and the remarks following it]{HENROT-e} we have $\partial\Omega\subset K$, it suffices to prove $K\subset \partial\Omega$. Suppose that there is a point $x\in K\setminus \partial\Omega$. Then $x$ must be in $\Omega$ or in $\overline{D}\setminus\overline{\Omega}$, and in either case we could find a small ball $B_x$ around $x$ whose closure $\overline{B}_x$ is contained in $\Omega$ or in $\mathbb{R}^n\setminus \overline{\Omega}$. By \cite[Proposition 2.2.17]{HENROT-e} and by the convergence in the sense of compacts, this implies $\overline{B}_x\subset \Omega_m$ for all sufficiently large $m$ in the first case and $\overline{B}_x\subset \mathbb{R}^n\setminus \overline{\Omega}_m$ in the second. Both cases contradict the fact that $x$ is the limit of a sequence of points $x_m\in \partial\Omega_m$. 
\end{proof}

\subsection{Shape admissible domains and compactness}\label{SubsecShapeAdm}

We prove a compactness result for suitable classes of domains with a quantitative control of the geometry of the domain and its boundary. 
This  
generalizes \cite[Theorem 2.4.10]{HENROT-e} and in part follows its ideas. 
Recall that we assume $n\geq 2$ throughout. 

\begin{definition}\label{DefShapeAdmis}
Let $D_0\subset D \subset \mathbb{R}^n$ be non-empty  bounded Lipschitz  
domains. A pair $(\Omega,\mu)$ is called a \emph{shape admissible domain} with parameters $D$, $D_0$, $\varepsilon>0$, $n-1\leq s<n$, $0\leq d\leq s$,   $\bar c_{s}^A>0$, $c_{d}^A>0$ if $\Omega$ is an $(\varepsilon,\infty)$-domain such that $D_0\subset  \Omega\subset D$ and $\mu$ is a boundary volume for $\Omega$ satisfying ~\eqref{E:upreg} and ~\eqref{E:lowreg-w} with $\partial \Omega$ in place of $K$. The set of such domains is denoted by $U_{ad}(D,D_0,\varepsilon, s, d, \bar c_{s}^A,c_{d}^A)$.
\end{definition}
Note that a pair $(\Omega,\mu)$ can be called a \emph{Jonsson shape admissible domain} with parameters $D$, $D_0$, $\varepsilon$, $s$, $d$, $c_s$, $c_d$, $\bar c_1$, $c_2$ if $D_0\subset  \Omega\subset D$ is a bounded $(\varepsilon,\infty)$-domain and $\mu$ is a boundary volume for $\Omega$ satisfying ~\eqref{EqMuDs}, ~\eqref{EqMuLd} and ~\eqref{Eqmeasurebounds} with $\partial \Omega$ in place of $K$.
The set of such domains is denoted by $U_{ad}^J(D,D_0,\varepsilon, s, d, c_s, c_d, \bar c_1, c_2)$. It is clear that this set is a closed, and hence compact, subset of $U_{ad}(D,D_0,\varepsilon, s, d, \bar c_{s}^A,c_{d}^A)$ in the sense of the following theorem. 

\begin{theorem}\label{T:compact}Suppose that the parameters are fixed in Definition~\ref{DefShapeAdmis}. 
\begin{enumerate}
\item[(i)] The class $U_{ad}(D,D_0,\varepsilon, s, d, \bar c_{s}^A,c_{d}^A)$ of admissible domains is  compact   in the Hausdorff sense, in the sense of characteristic functions, in the sense of compacts, and in the sense of weak convergence of the boundary volumes. 
\item[(ii)] If  for a sequence $(\Omega_m)_{m }$  of shape admissible domains  the boundary volumes 
converge weakly,
then  $(\Omega_m)_{m}$ converges  in the Hausdorff sense, in the sense of characteristic functions, and in the sense of compacts.  
\end{enumerate}
\end{theorem}

\begin{proof}
By ~\eqref{E:upreg} the measures are uniformly bounded, so that by~Theorem~\ref{T:nothing_extra} (i) we can find a subsequence $(m_k)_k$, an open set $\Omega$ to which the domains $\Omega_{m_k}$ converge in the Hausdorff sense and in the sense of characteristic functions and a  Borel measure $\mu$ with support $K=\supp \mu$ satisfying ~\eqref{E:upreg} and ~\eqref{E:lowreg-w} and such that $\mu_{m_k}\to \mu$ weakly. By~Theorem~\ref{T:epsinftystable} the open set $\Omega$ is a bounded $(\varepsilon,\infty)$-domain. Since $D_0$ is a subset of all $\Omega_{m_k}$'s, it is a subset of $\Omega$, which therefore is seen to be non-empty.

It remains to show convergence in the sense of compacts. If $L$ is a compact subset of $\Omega$, then by \cite[Proposition 2.2.17]{HENROT-e} $L$ is contained in $\Omega_m$ for any large enough $m$. Now suppose that $L\subset \mathbb{R}^n\setminus\overline{\Omega}$, we may assume that $L$ has non-empty interior. Suppose that there is a subsequence $(\Omega_{m_k})_k$ such that $L\cap \overline{\Omega}_{m_k}\neq \emptyset$ for all $k$. If (for some subsequence) we have $L \subset \Omega_{m_k}$, then $\lambda^n(\Omega_{m_k}\setminus \Omega)\geq \lambda^n(L)>0$, what contradicts the convergence in the sense of characteristic functions. If this is not the case, then we must have $L\cap \partial\Omega_{m_k}\neq \emptyset$ (for some subsequence). However, this also leads to a contradiction: Write $\gamma:=d(L,\overline{\Omega})$. Consider a sequence of points $x_{m_k}\subset L\cap \partial\Omega_{m_k}$ converging to a point $x\in L$. For large enough $k$ we have $x_{m_k}\in B(x,\gamma/2)$. Since $x_{m_k}\in \partial\Omega_{m_k}$ we can find $y_{m_k}\in \Omega_{m_k}$ such that $|x_{m_k}-y_{m_k}|<2^{-k}$, and passing to a subsequence if necessary, we may assume the $y_{m_k}$ converge to a point $y\in \overline{B(x,\gamma/2)}$. Fix a point $z\in \Omega$. Then we have $z\in \Omega_{m_k}$ for all sufficiently large $k$ by \cite[Proposition 2.2.17]{HENROT-e}, and for each such $k$ we can find a rectifiable curve $\gamma_k$ joining $z$ and $y_{m_k}$ such that $C(\gamma_k,\varepsilon)\subset \Omega_{m_k}$. Passing to another subsequence if necessary we may, by Lemma~\ref{L:cigar}, assume that the curves $\gamma_k$ converge to a rectifiable curve $\gamma$ joining $z$ and $y$ and that the sets $C(\gamma_k,\varepsilon)$ converge to $C(\gamma,\varepsilon)$ in the Hausdorff sense. Since $C(\gamma,\varepsilon)\subset \Omega$ by \cite[(2.16) in 2.2.3.2]{HENROT-e} and $y\in \overline{C(\gamma,\varepsilon)}\subset \overline{\Omega}$, this implies that $d(x,\overline{\Omega})\leq \gamma/2$, what contradicts the fact that $x\in L$. Consequently we have $L\subset \mathbb{R}^n\setminus \Omega_m$ for all sufficiently large $m$, and can conclude that $\Omega_m\to\Omega$ in the sense of compacts.
From~Theorem~\ref{T:nothing_extra} (ii) it now follows that $K=\partial\Omega$. This proves (i).

To see (ii), note that by Proposition~\ref{prop-A} or Lemma~\ref{measure-lemma}, respectively, $\mu$ satisfies the  desired scaling conditions.  By (i) and Theorem~\ref{T:epsinftystable} $(\Omega_m)_{m}$ has subsequence   convergent in the Hausdorff sense, in the sense of characteristic functions and in the sense of compacts to some $\Omega\subset D$ which is an $(\varepsilon,\infty)$-domain contained in $D$ and 
such that $\supp\mu=\partial\Omega$. Since the limit domain of any such subsequences of domains must have this boundary but at the same time be bounded, $\Omega$ is the limit   of the sequence. 
\end{proof}
%
%

\begin{remark}
Theorem~\ref{T:nothing_extra} and Theorem~\ref{T:compact} rely on ~\eqref{E:lowreg-w} in a crucial way. For $d=0$ in Definition~\ref{DefShapeAdmis} estimate ~\eqref{E:upreg} reduces to a uniform bound for the total masses, but this already suffices to have Theorem~\ref{T:nothing_extra} and Theorem~\ref{T:compact}. 
\end{remark}

\section{Trace and extension operators}
\label{S:traceExt}

We review known trace and extension methods that combine well with our setup and record some consequences.  

\subsection{Traces and extensions for closed subsets of $\mathbb{R}^n$} 

For any $\beta>0$ the symbol $H^\beta(\mathbb{R}^n)$ denotes the Bessel-potential space of order $\beta$, \cite{AH96, Triebel78, TRIEBEL-1997}, that is, the space of all $f\in L^2(\mathbb{R}^n)$ such that $(1+|\xi|^2)^{\beta/2}\hat{f}\in L^2(\mathbb{R}^n)$, where $f\mapsto \hat{f}$ denotes the Fourier transform. It is a Hilbert space with norm $\|f\|_{H^\beta(\mathbb{R}^n)}:=\|(1+|\xi|^2)^{\beta/2}\hat{f}\|_{L^2(\mathbb{R}^n)}$.

We are interested in the trace of functions $f\in H^\beta(\mathbb{R}^n)$ to a closed set $K\subset \mathbb{R}^n$. For $\beta >n/2$ we have $H^\beta(\mathbb{R}^n)\subset C(\mathbb{R}^n)$, see e.g. \cite[2.8.1]{Triebel78},
and one can use the pointwise restriction $f|_K$. However, for our purposes the case $0<\beta\leq n/2$ is relevant. It is well known that for any $f\in H^\beta(\mathbb{R}^n)$ the limit 
\begin{equation}\label{E:traceop}
\widetilde{f}(x)=\lim_{r\to 0}\frac{1}{\lambda^n(B(x,r))}\int_{B(x,r)}f(y)dy
\end{equation}
exists at $H^\beta(\mathbb{R}^n)$-quasi every $x\in\mathbb{R}^n$ and that $\widetilde{f}$ defines a $H^\beta(\mathbb{R}^n)$-quasi continuous version of $f$, \cite[Theorem 6.2.1]{AH96}. If $\mu$ is a Borel measure with support $K=\supp\mu$ and satisfying  ~\eqref{E:upreg} for sufficiently large $d$, then it charges no set of zero $H^\beta(\mathbb{R}^n)$-capacity and consequently the limit in ~\eqref{E:traceop} does exist for all $x\in K\setminus N$, where $N\subset K$ is a $\mu$-null set. Under these circumstances one can define a $\mu$-class $\operatorname{Tr}_K f$ on $K$ by setting
\begin{equation}\label{E:traceopdef}
\operatorname{Tr}_K f(x)=\widetilde{f}(x)\quad \text{if $x\in K\setminus N$ and $0$ otherwise}. 
\end{equation}
We state a direct consequence of \cite[Theorems 7.2.2 and 7.3.2, together with Propositions 5.1.2 and 5.1.4]{AH96}.

\begin{theorem}\label{ThTracecheap}
	Let  $0< d\leq n$, $(n-d)/2<\beta\leq n/2$ and $2<q<2d/(n-2\beta)$ (with $1/0:=+\infty$). Suppose that $K=\supp\mu$ with a Borel measure $\mu$ satisfying ~\eqref{E:upreg}. Then $\operatorname{Tr}_K$ is a compact linear operator from $H^\beta(\R^n)$ into $L^q(K,\mu)$, and we have $\left\|\operatorname{Tr}_K f\right\|_{L^q(K)}\leq c_{\operatorname{Tr}}\left\|f\right\|_{H^\beta(\mathbb{R}^n)}$, $f\in H^\beta(\mathbb{R}^n)$, with a constant $c_{\mathrm{Tr}}>0$ depending only on $\beta$, $d$, $n$, $c_d^A$, $q$.
\end{theorem}

In the situation of this theorem we have $H^\beta(\mathbb{R}^n)=\mathring{H}^\beta(\mathbb{R}^n\setminus K)\oplus \mathcal{H}_K$, where $\mathring{H}^\beta(\mathbb{R}^n\setminus K)$ is the closure of $C_c^\infty(\mathbb{R}^n\setminus K)$ in $H^\beta(\mathbb{R}^n)$ and $\mathcal{H}_K$ denotes its orthogonal complement, \cite[Corollary 2.3.1 and Lemma 2.3.4]{FOT94}. 
Given $\varphi\in \mathrm{Tr}_K(H^\beta(\mathbb{R}^n))$ we say that $g\in H^\beta(\mathbb{R}^n)$ is a weak solution to the Dirichlet problem 
\begin{equation}\label{E:FDP}
(1-\Delta)^\beta g=0 \quad \text{on $\mathbb{R}^n\setminus K$},\quad \mathrm{Tr}_Kg=\varphi \quad \text{on $K$}
\end{equation}
if $\left\langle g,v\right\rangle_{H^\beta(\mathbb{R}^n)}=0$, $v\in \mathring{H}^\beta(\mathbb{R}^n\setminus K)$, and $\mathrm{Tr}_Kg=\varphi$ $\mu$-a.e. on $K$. The following is folklore, see for instance \cite{FarkasJacob2001, TRIEBEL-1997}.

\begin{corollary}\label{C:harmonicext} Let the hypotheses of Theorem~\ref{ThTracecheap} be in force.
	\begin{enumerate}
		\item[(i)] For any $\varphi \in \mathrm{Tr}_K(H^\beta(\mathbb{R}^n))$ there is a unique weak solution $H_{K}\varphi$  to ~\eqref{E:FDP}. 
		\item[(ii)] The map $\varphi \mapsto \left\|\varphi\right\|_{\mathrm{Tr}_K(H^\beta(\mathbb{R}^n))}:=\inf_{g\in H^\beta(\mathbb{R}^n), \varphi=\mathrm{Tr}_K g} \left\|g\right\|_{H^\beta(\mathbb{R}^n)}$ is a norm that makes $\mathrm{Tr}_K(H^\beta(\mathbb{R}^n))$ a Hilbert space. 
		\item[(iii)] The map $H_{K}:\mathrm{Tr}_K(H^\beta(\mathbb{R}^n))\to H^\beta(\mathbb{R}^n)$, $\varphi\mapsto H_{K}\varphi$, is a linear extension operator of norm one, and $\mathrm{Tr}_K(H_{K}\varphi)=\varphi$, $\varphi \in \mathrm{Tr}_K(H^\beta(\mathbb{R}^n))$.
	\end{enumerate}
\end{corollary}

To the linear operator, $H_K$ one also refers to \emph{$1$-harmonic extension operator}. 
\begin{proof}
	If $\varphi=\mathrm{Tr}_K f$ with $f\in H^\beta(\mathbb{R}^n)$, then the orthogonal projection $H_{K}\varphi$ of $f$ onto $\mathcal{H}_K$ has the desired properties, \cite[Section 2.3]{FOT94}. The rest follows.
\end{proof}

\begin{remark}
A description of the space $\mathrm{Tr}_K(H^\beta(\mathbb{R}^n))$ in terms of an atomic decomposition is provided in \cite{JONSSON-2009}. Note that for orders $p$ of integrability other than $2$, the $1$-harmonic extension is generally no longer linear.
\end{remark}

For later use, we record the following convergence result for integrals of traces.

\begin{theorem}\label{L:traceconvergence}  
Let $D\subset\mathbb{R}^n$ be a bounded open set, $0< d\leq n$ and $(n-d)/2<\beta\leq n/2$. Let $(\mu_m)_m$ be a sequence of finite Borel measures with supports $K_m=\supp\mu_m$ contained in $\overline{D}$ and such that ~\eqref{E:upreg} holds for all $m$ with the same constant. Suppose that $(\mu_m)_m$ converges weakly to a Borel measure $\mu$. If $(v_m)_m\subset H^{\beta}(\mathbb{R}^n)$ is a sequence that converges to some $v$ in $H^{\beta}(\mathbb{R}^n)$, then 
	\[\lim_{m\to\infty}\int_{K_{m}}|\mathrm{Tr}_{K_m}v_{m}|^2 d\mu_{m} =\int_{K}|\mathrm{Tr}_{K} v|^2  d\mu,\]
	where $K:=\supp\mu$.  
\end{theorem}

\begin{proof}
	By Lemma~\ref{measure-lemma} also $\mu$ satisfies ~\eqref{E:upreg} and Theorem~\ref{ThTracecheap} applies. Since $C^\infty_c(\mathbb{R}^n)$ is dense in $H^\beta(\mathbb{R}^n)$, we can find a sequence $(\varphi_j)_{j\in \N} \subset C_c^\infty(\mathbb{R}^n)$ converging to $v$ in $H^\beta(\mathbb{R}^n)$. Following \cite{CAPITANELLI-2010-1}, we observe that
	\begin{multline}
	\left|\int_{K_m}|\mathrm{Tr}_{K_m} v_{m}|^2 d\mu_{m} -\int_{K}|\mathrm{Tr}_{K} v|^2  d\mu \right|\\
	\le \left|\int_{K_{m}}|\mathrm{Tr}_{K_{m}} v_{m}|^2  d\mu_{m} -\int_{K_{m}}|\mathrm{Tr}_{K_{m}} v|^2  d\mu_{m}\right|
	+\left|\int_{K_{m}}|\mathrm{Tr}_{K_{m}} v|^2  d\mu_{m} -\int_{K_{m}} |\varphi_{j}|^2  d\mu_{m}\right| \\
	+\left|\int_{K_{m}} |\varphi_{j}|^2  d\mu_{m} -\int_{K} |\varphi_{j}|^2  d\mu\right|
	+\left|\int_{K} |\varphi_{j}|^2  d\mu -\int_{K}|\mathrm{Tr}_{K} v|^2  d\mu\right|.\label{EqBigEstInt-}
	\end{multline}
	To estimate the first term on the right hand side of ~\eqref{EqBigEstInt-} we control it using the Cauchy-Schwarz inequality and the reverse triangle inequality, 
	\begin{multline}
	\left|\int_{K_{m}}|\mathrm{Tr}_{K_{m}} v_{m}|^2  d\mu_{m} -\int_{K_{m}}|\mathrm{Tr}_{K_{m}} v|^2  d\mu_{m}\right|\notag\\
	\le \|\mathrm{Tr}_{K_{m}} (v_{m} - v)\|_{L^2(K_{m})}\left(\|\mathrm{Tr}_{K_{m}}v_{m}\|_{L^2(K_{m})}+\|\mathrm{Tr}_{K_{m}} v\|_{L^2(K_{m})}\right).
	\end{multline}
	Since $\beta$, $d$, $n$ and $c_d^A$ are kept fixed and $\sup_m \mu_m(\overline{D})<+\infty$ by weak convergence, 
	Theorem~\ref{ThTracecheap} and H\"older's inequality ensure the existence of a constant $c_{\mathrm{Tr}}'>0$, independent of $m$, such that 
	$\|\mathrm{Tr}_{K_m} (v_m-v)\|_{L^2(K_m)}\le c_{\mathrm{Tr}}' \|v_m-v\|_{H^\beta(\R^n)}$, 
	what goes to zero as $m\to \infty$. Since also 
	\[\max \{\sup_m \|\mathrm{Tr}_{K_{m}}v_{m}\|_{L^2(K_{m})}, \sup_m \|\mathrm{Tr}_{K_{m}}v\|_{L^2(K_{m})}\} \leq c_{\mathrm{Tr}}'\sup_m \|v_{m}\|_{H^\beta(\R^n)},\]
	the first term in~~\eqref{EqBigEstInt-} is seen to converge to $0$ as $m \to +\infty$. 
	For the second term in~~\eqref{EqBigEstInt-} we can use $\sup_m\|\mathrm{Tr}_{K_{m}} (v-\varphi_{j})\|_{L^2(K_{m})}\le c_{\mathrm{Tr}}'\| v-\varphi_{j}\|_{H^\beta(\R^n)}$
	to see it converges to zero as $j\to \infty$, and the same with $K$ in place of $K_{m}$ yield the convergence to zero of the last term. The third term converges to zero as $m\to\infty$ by weak convergence.
\end{proof}

If refined scaling properties of $\mu$ as in Subsection~\ref{SS:refined} are known, one can introduce Besov spaces on $K$ with explicit norms, see \cite{JONSSON-1984-1, JONSSON-1994, JONSSON-2009}. We recall the definition given initially in \cite{JONSSON-1994}.

\begin{definition}\label{D:Besov} 
Let $0\leq d\leq n$, $d\leq s\leq n$, $s>0$ and $(n-d)/2<\beta <1+(n-s)/2$. Suppose $\mu$ is a Borel measure on $\mathbb{R}^n$ with support $\supp \mu=K$ satisfying ~\eqref{EqMuDs}, ~\eqref{EqMuLd},  
~\eqref{Eqnormalized}.
The Besov spaces $B_\beta^{2,2}(K,\mu)$ on $K$ is defined as the space of $\mu$-classes of real-valued functions $f$ on $K$ such that the norm 
	\begin{multline}
	\left\|f\right\|_{B_\beta^{2,2}(K,\mu)}:=\notag\\
	\left\|f\right\|_{L^2(K,\mu)}+\left(\sum_{j=0}^\infty 2^{j(\beta-\frac{n}{2})}\int\int_{|x-y|<2^{-j}}\frac{(f(x)-f(y))^2}{\mu(B(x,2^{-j}))\mu(B(y,2^{-j}))}\mu(dy)\mu(dx)\right)^{1/2}
	\end{multline}
	is finite.
\end{definition}

The spaces $B_\beta^{2,2}(K,\mu)$ are Hilbert spaces. If $\mu_1$ and $\mu_2$ are two different measures satisfying the hypotheses of  Definition~\ref{D:Besov} and with the same support $K$, then  Theorem~\ref{ThTrace} below implies that the resulting spaces $B_\beta^{2,2}(K,\mu_1)$ and $B_\beta^{2,2}(K,\mu_2)$ are equivalent Hilbert spaces, see \cite[Section 3.5]{JONSSON-1994}. We therefore simply write $B^{2,2}_\beta(K)$ for $B_\beta^{2,2}(K,\mu)$. The following result is a special case of \cite[Theorem 1]{JONSSON-1994}.

\begin{theorem}\label{ThTrace}
Let  $0\leq d\leq n$, $d\leq s\leq n$, $s>0$ and $(n-d)/2<\beta <1+(n-s)/2$. 
Suppose $K\subset \R^n$ is a closed set which is the support of a Borel measure $\mu$ satisfying ~\eqref{EqMuDs}, ~\eqref{EqMuLd},  
~\eqref{Eqnormalized}.
Then  
	\begin{enumerate}
		\item[(i)] $\operatorname{Tr}_K$ is a continuous linear operator from $H^\beta(\R^n)$ onto $B^{2,2}_\beta(K)$, and there is 
		a constant $c_{\mathrm{Tr}}>0$ depending only on $\beta$, $s$, $d$, $n$, $c_s$, $c_d$, $c_1$ and $c_2$ such that 
		$\left\|\operatorname{Tr}_K f\right\|_{B_\beta^{2,2}(K)}\leq c_{\operatorname{Tr}}\left\|f\right\|_{H^\beta(\mathbb{R}^n)}$, $f\in H^\beta(\mathbb{R}^n)$.
		\item[(ii)] There is a continuous linear extension operator $E_K:B^{2,2}_\beta(K)\to  H^\beta (\R^n)$ such that $\operatorname{Tr}_K(E_K f)=f$, $f\in B^{2,2}_\beta(K)$.
	\end{enumerate}
\end{theorem}

The independence of the constant $c_{\mathrm{Tr}}$ of all except the stated quantities follows from \cite[Lemma 3 and its proof]{JONSSON-1994}.

\subsection{$W^{1,2}$-admissible domains}\label{Ss:SAD}
We define a class of domains well adapted to boundary value problems and prove basic facts about associated trace and extension operators. We remind the reader that we assume $n\geq 2$ throughout.

Recall that the Sobolev space $W^{k,p}(\Omega)$ with $k\in \N$ and $p\in [1,\infty)$ is defined as the space of all $f\in L^p(\Omega)$ for which we have $D^\gamma f\in L^p(\Omega)$ in the distributional sense for any multi-index $\gamma$ satisfying $|\gamma|\le k$. It is well known and easy to see that for nonnegative integers $k$ the space $H^k(\mathbb{R}^n)$ coincides with the Sobolev space $W^{k,2}(\mathbb{R}^n)$ in the sense of equivalently normed vector spaces. 

Given $k=1,2,...$ and $1\leq p\leq \infty$, a domain $\Omega\subset\mathbb{R}^n$ is called a \emph{$W^{k,p}$ -extension domain} if there exists a bounded linear extension operator $E: W^{k,p}(\Omega)\to W^{k,p}(\mathbb{R}^n)$, \cite[p. 1218]{HAJLASZ-2008}. 
Every Lipschitz domain is a $W^{k,p}$ -extension domain for any $k=1,2,...$ and $1\leq p\leq \infty$, see \cite{CALDERON-1961, STEIN-1970}. It was shown in \cite{JONES-1981} that any \emph{$(\eps,\delta)$-domain} $\Omega\subset \mathbb{R}^n$, i.e., any (possibly unbounded) domain $\Omega\subset \mathbb{R}^n$ satisfying the conditions (i) and (ii) in Definition~\ref{DefEDD} for all $x,y\in \Omega$ with $|x-y|<\delta$ for some fixed $\delta>0$, is a $W^{k,p}$-extension domain for any $k=1,2,...$ and $1\leq p\leq \infty$, see also \cite{AHMNT,Rogers}. In particular, we have the following.

\begin{corollary}
	Every $(\varepsilon,\infty)$-domain is a $W^{k,p}$-extension domain for any $k=1,2,...$ and $1\leq p\leq \infty$, and therefore also every shape admissible domain in the sense of Definition~\ref{DefShapeAdmis}.
\end{corollary}

Any Lipschitz domain is an $(\varepsilon,\delta)$-domain for some $\varepsilon$ and $\delta$, \cite{JONES-1981}, and any bounded Lipschitz domain is an $(\eps,\infty)$-domain for some suitable $\varepsilon>0$. For $n\geq 3$ examples of $W^{1,p}$-extension domains are known which are no $(\eps,\delta)$-domains,~\cite{JONES-1981}.

We quote an extension result for Bessel-potential spaces on domains $\Omega$. For $\beta>0$ we write $H^\beta(\Omega)=\{f\in D'(\Omega): \text{$f=g|_\Omega$ for some $g\in H^\beta(\mathbb{R}^n)$}\}$. Endowed with the norm defined by $\left\|u\right\|_{H^\beta(\Omega)}=\inf_{g\in H^\beta(\mathbb{R}^n), f=g|_\Omega}\left\|g\right\|_{H^\beta(\mathbb{R}^n)}$ it becomes a Hilbert space. It follows from this definition that for $W^{1,2}$-extension domains $\Omega\subset \mathbb{R}^n$ the spaces $H^1(\Omega)$ and $W^{1,2}(\Omega)$ agree as equivalently normed Hilbert spaces, see \cite[4.2.1 and 4.2.4]{Triebel78} for a more classical case. The following will be used in the next section.

\begin{proposition}\label{P:extensionfracSobo}
	Let $\Omega\subset\mathbb{R}^n$ be a bounded $(\varepsilon,\infty)$-domain. Then there is a linear extension operator $f\mapsto \ext_\Omega f$ such that for any $0\leq \beta\leq 1$ we have $\ext_\Omega:H^\beta(\Omega)\to H^\beta(\mathbb{R}^n)$ with 
	$\left\|\ext_\Omega f\right\|_{H^\beta(\mathbb{R}^n)}\leq c_{\ext}\left\|f\right\|_{H^\beta(\Omega)}$, $f\in H^\beta(\Omega)$,
	with a constant $c_{\ext}>0$ depending only on $n$, $\varepsilon$ and $\beta$. 
\end{proposition} 
\begin{proof}
	As in \cite[Theorem 5.8]{CAPITANELLI-2010} this proposition follows from \cite[Theorem 8]{Rogers} and the fact that $H^\beta(\Omega)$ can be obtained by interpolation from $L^2(\Omega)$ and $H^1(\Omega)$, to see this one can follow the arguments used to prove \cite[Theorem 2.13]{Triebel2002}.
\end{proof} 

\begin{remark}\label{R:LukeandJones}
	For a bounded $(\varepsilon,\infty)$-domain $\Omega$ the existence of a bounded linear extension operator from $W^{1,2}(\Omega)$ to $W^{1,2}(\mathbb{R}^n)$ with norm bound depending only on $\varepsilon$ and $n$ follows from \cite[Theorem 1]{JONES-1981}. However, \cite[Theorem 8]{Rogers} allows to use one and the same extension operator for different spaces $W^{k,p}$, what allows interpolation. 
\end{remark} 

We have the following partial generalization of results from \cite{ARFI-2017,ROZANOVA-PIERRAT-2020} and \cite{FarkasJacob2001} on embeddings and trace and extension operators and their compactness.

\begin{theorem}\label{ThGENERIC}
	Let $\Omega$ be a $W^{1,2}$-extension domain. 
	\begin{enumerate}
		\item[(i)] The space $W^{1,2}(\Omega)$ is compactly embedded in $L^2_{\loc}(\Omega)$ (or in $L^2(\Omega)$ if $\Omega$ is bounded). The linear operator $\mathrm{Tr}_\Omega: W^{1,2}(\R^n)\to W^{1,2}(\Omega)$, $\mathrm{Tr}_\Omega f=f|\Omega$, is bounded and has a linear bounded right inverse $E_\Omega: W^{1,2}(\Omega)\to W^{1,2}(\R^n)$. 
		\item[(ii)] Let $\mu$ be a Borel measure with compact support $\supp\mu=\Gamma\subset \overline{\Omega}$ which satisfies ~\eqref{E:upreg} with some $n-2<d\leq n$. Then the operator $\mathrm{Tr}_{\Gamma}: W^{1,2}(\R^n)\to L^2(\Gamma,\mu)$, defined by ~\eqref{E:traceopdef}, is compact. The operator 
		\[\mathrm{Tr}_{\Omega, \Gamma}:=\mathrm{Tr}_{\Gamma}\circ E_{\Omega}:\ W^{1,2}(\Omega)\to L^2(\Gamma,\mu) \] is well defined in the sense that if $u,v\in W^{1,2}(\mathbb{R}^n)$ are such that $u=v$ $\lambda^n$-a.e. in $\Omega$, then $\mathrm{Tr}_{\Gamma}u=\mathrm{Tr}_{\Gamma}v$ $\mu$-a.e. on $\Gamma$, and it is compact. The image $\mathrm{Tr}_{\Omega, \Gamma}(W^{1,2}(\Omega))=\mathrm{Tr}_{\Gamma}(W^{1,2}(\mathbb{R}^n))$ is dense in $L^2(\Gamma,\mu)$. The map 
		\[\varphi\mapsto \left\|\varphi\right\|_{\mathrm{Tr}_{\Gamma}(W^{1,2}(\mathbb{R}^n))}:=\inf_{g\in W^{1,2}(\mathbb{R}^n), \varphi=\mathrm{Tr}_\Gamma g} \left\|g\right\|_{W^{1,2}(\mathbb{R}^n)}\]
		defines a Hilbert norm on $\mathrm{Tr}_{\Gamma}(W^{1,2}(\mathbb{R}^n))$ with respect to which both operators have linear bounded right inverses $H_{\Gamma}: \mathrm{Tr}_{\Gamma}(W^{1,2}(\mathbb{R}^n))\to W^{1,2}(\R^n)$ respectively 
		\[H_{\Gamma,\Omega}:=\mathrm{Tr}_\Omega\circ H_{\Gamma}:\: \mathrm{Tr}_{\Gamma}(W^{1,2}(\mathbb{R}^n))\to W^{1,2}(\Omega).\]
		\item[(iii)] Suppose that $\partial\Omega$ is compact and $\mu$ is a Borel measure with $\supp\mu=\partial\Omega$ which satisfies ~\eqref{E:upreg} with some $n-2<d\leq n$. For all $u\in W^{1,2}(\Omega)$ with $\Delta u\in L^2(\Omega)$ we can define 
		a bounded linear functional $\frac{\del u}{\del n} \in (\mathrm{Tr}_{\partial\Omega}(W^{1,2}(\mathbb{R}^n)))'$ by
		\begin{equation}\label{FracGreen}
		\langle \frac{\del u}{\del n}, 
		\mathrm{Tr}_{\Omega,\partial\Omega} v\rangle_{(\mathrm{Tr}_{\partial\Omega}(W^{1,2}(\mathbb{R}^n)))',\mathrm{Tr}_{\partial\Omega}(W^{1,2}(\mathbb{R}^n))}=\int_\Omega v\Delta u\dx + \int_\Omega \nabla v\cdot \nabla u \dx, 
		\end{equation}
		$v\in W^{1,2}(\Omega)$. Similarly, for any $u\in W^{1,2}(\Omega)$ and $1\leq i\leq n$, we can define a bounded linear functional $u\cdot n_i\in (\mathrm{Tr}_{\partial\Omega}(W^{1,2}(\mathbb{R}^n)))'$ by  
		\begin{equation}\label{FracGreen0}
		\langle u\cdot n_i, 
		\mathrm{Tr}_{\Omega,\partial\Omega} v\rangle_{(\mathrm{Tr}_{\partial\Omega}(W^{1,2}(\mathbb{R}^n)))',\mathrm{Tr}_{\partial\Omega}(W^{1,2}(\mathbb{R}^n))}=\int_\Omega \frac{\partial u}{\partial x_i} v \dx+\int_\Omega u \frac{\partial v}{\partial x_i} \dx, 
		\end{equation}
		$v\in W^{1,2}(\Omega)$.       
	\end{enumerate}
\end{theorem}

\begin{remark}
	The distribution $\frac{\del u}{\del n}$ in ~\eqref{FracGreen} is a generalized normal derivative. 
\end{remark}

\begin{remark}
Even if $\partial\Omega$ is Lipschitz it can make sense to endow it with a measure $\mu$ that satisfies ~\eqref{E:upreg} with maximal possible exponent $n-2<d<n-1$, in this case one allows $\mu$ to have parts singular w.r.t. $\mathcal{H}^{n-1}$.
\end{remark}

\begin{remark}
If $n-2<d\leq s<n$ and $\mu$ satisfies ~\eqref{EqMuDs}, ~\eqref{EqMuLd} and ~\eqref{Eqnormalized}, then by Theorem~\ref{ThTrace} the space $\mathrm{Tr}_{\Gamma}(W^{1,2}(\mathbb{R}^n))$ and the operator $H_\Gamma$ in Theorem~\ref{ThGENERIC} (ii) can be replaced by $B_1^{2,2}(\Gamma)$ and $E_\Gamma$. Under these more restrictive hypotheses Theorem~\ref{ThGENERIC} (iii) holds with $B_1^{2,2}(\partial\Omega)$ in place of $\mathrm{Tr}_{\partial\Omega}(W^{1,2}(\mathbb{R}^n)$, and the same replacement can be made in Corollary~\ref{C:projectright} and Corollary~\ref{C:Dirichletdata} below.
\end{remark}

\begin{proof}
	Statement (i) is a special case of Theorem~2.12 point 2 in~\cite{ARFI-2017}, which generalizes the classical Rellich-Kondrachov theorem. Note that since $\Omega$ is $W^{1,2}$-extension domain, $\Omega$ is an $n$-set and $W_2^1(\Omega)=C_2^1(\Omega)$, see \cite[Theorem~5]{HAJLASZ-2008}, and this is sufficient to conclude the mentioned result in \cite{ARFI-2017}. The first statement (ii) follows from Theorem~\ref{ThTracecheap} and the finiteness of $\mu$. That $\mathrm{Tr}_{\Omega, \Gamma}:W^{1,2}(\Omega)\to L^2(\Gamma,\mu)$ in (ii) is well defined in the stated sense can be seen as in \cite[Theorem 1]{WALLIN-1991} or \cite[Theorem 6.1]{Biegert2009}. Its compactness follows from \cite[Corollary 7.4]{Biegert2009} (see also \cite[Proposition 3]{ROZANOVA-PIERRAT-2020}). The space $\{v|_{\del \Omega} : v \in C_c^\infty(\R ^n)\}$, is uniformly dense in $C(\del \Omega)$ by the Stone-Weierstrass theorem, and $C_c^\infty(\R ^n)$ is dense in $W^{1,2}(\mathbb{R}^n)$, hence $\mathrm{Tr}_{\Gamma}(W^{1,2}(\mathbb{R}^n)$ is dense in $L^2(\Gamma,\mu)$. The last statements follow using Corollary~\ref{C:harmonicext}. For (iii) one can follow the  arguments of \cite[Proposition~1]{ARFI-2017} (originally due to \cite[Theorem.~4.15]{LANCIA-2002}), the correctness of the definition can be concluded using \cite[formula (2.3.7) in Section 2.3]{FOT94}. In a similar manner one can obtain ~\eqref{FracGreen0}, see \cite[Theorem 2.5 and formula (2.17)]{GR86} for the Lipschitz case. 
\end{proof}

Theorem~\ref{ThGENERIC} (iii) and \cite[Definition 7]{ARFI-2017} motivate to define a class of domains suitable to discuss different types of boundary value problems (see also~\cite{ROZANOVA-PIERRAT-2020,DEKKERS-2020}).

\begin{definition} 
	
	\label{DefSAdmis}
	A \emph{$W^{1,2}$-admissible domain} in $\mathbb{R}^n$ is a pair $(\Omega,\mu)$ consisting of a $W^{1,2}$-extension domain $\Omega\subset \R^n$ and a Borel measure $\mu$ with $\supp\mu=\partial\Omega$ which satisfies ~\eqref{E:upreg} with some $n-2<d\leq n$. We call a $W^{1,2}$-admissible domain $(\Omega,\mu)$ bounded if $\Omega$ is bounded.
\end{definition}

Examples of $W^{1,2}$-admissible domains are $C^k$-regular domains ($k\in \N^*$), Lipschitz domains and domains with a $d$-set boundary ($n-2<d<n$) or a boundary composed of different $d$-sets such as the cylindrical von Koch domains in~\cite{LANCIA-2003,LANCIA-2010}.

To discuss boundary value problems on $W^{1,2}$-admissible domains $(\Omega,\mu)$ it is useful to consider $\mathrm{Tr}_{\partial\Omega}(W^{1,2}(\mathbb{R}^n))$ with equivalent scalar products. For any $\varphi \in \mathrm{Tr}_{\partial\Omega}(W^{1,2}(\mathbb{R}^n))$ the function $H_{\partial\Omega,\Omega}(\varphi)\in W^{1,2}(\Omega)$ is the unique minimizer for the Dirichlet energy $\int_\Omega |\nabla v|^2dx$ in the class of all $v\in W^{1,2}(\Omega)$ with $\mathrm{Tr}_{\Omega,\partial\Omega} v=\varphi$ $\mu$-a.e. on $\partial\Omega$. By $\left\|\cdot\right\|_{\mathrm{Tr}_{\partial\Omega}(W^{1,2}(\mathbb{R}^n))}$ we denote the scalar product on $\mathrm{Tr}_{\partial\Omega}(W^{1,2}(\mathbb{R}^n))$ associated with the Hilbert norm in Theorem~\ref{ThGENERIC} (ii) with $\Gamma=\partial\Omega$.

\begin{corollary}\label{C:projectright}
	Let $(\Omega,\mu)$ be a bounded $W^{1,2}$-admissible domain in $\mathbb{R}^n$ and let $\gamma$ be a nonnegative and bounded Borel function on $\partial\Omega$ which is positive on a subset positive $\mu$-measure. Then the bilinear form
	\begin{equation}\label{E:altspB}
	\left\langle \varphi,\psi\right\rangle_{\mathrm{Tr}_{\partial\Omega}(W^{1,2}(\mathbb{R}^n)),\gamma}:=\int_{\Omega}\nabla H_{\partial\Omega,\Omega}(\varphi)\nabla H_{\partial\Omega,\Omega}(\overline{\psi})dx+\int_{\partial\Omega}\gamma\varphi\overline{\psi}d\mu
	\end{equation}
	is an equivalent scalar product on $\mathrm{Tr}_{\partial\Omega}(W^{1,2}(\mathbb{R}^n))$. There is a constant $c>0$ depending only on $d$, $n$, $c_d^A$, the total mass of $\mu$ and $\gamma$ such that
	$\left\|\varphi\right\|_{\mathrm{Tr}_{\partial\Omega}(W^{1,2}(\mathbb{R}^n)),\gamma}\leq c\left\|\varphi\right\|_{\mathrm{Tr}_{\partial\Omega}(W^{1,2}(\mathbb{R}^n))}$, $\varphi\in \mathrm{Tr}_{\partial\Omega}(W^{1,2}(\mathbb{R}^n))$.
\end{corollary}

\begin{proof} 
	Well known arguments, see \cite[Theorem 21A and Step 3 in its proof on p. 247/248]{Zeidler}, together with Theorem~\ref{ThTracecheap} show that the bilinear form 
	\[\left\langle w,v\right\rangle_{W^{1,2}(\Omega),\gamma}:=\int_{\Omega}\nabla w\nabla \overline{v}\:dx+\int_{\partial\Omega}\gamma\operatorname{Tr}_{\Omega,\partial\Omega}w \operatorname{Tr}_{\Omega,\partial\Omega}\overline{v}d\mu,\quad v,w\in W^{1,2}(\Omega),\]
	is an equivalent scalar product on $W^{1,2}(\Omega)$, and with another application of Theorem~\ref{ThTracecheap} this implies the result.
\end{proof}

We complement Theorem~\ref{ThGENERIC} by results involving a Dirichlet boundary condition, they will be used in Section~\ref{S:shapeOp}. Suppose that $(\Omega,\mu)$ be a bounded $W^{1,2}$-admissible domain and $\Gamma_{\mathrm{Dir}}\subset \partial\Omega$ is a set of positive $\mu$-measure. Then 
\begin{equation}\label{EqVOM}
V(\Omega,\Gamma_{\mathrm{Dir}}):= \{w\in W^{1,2}(\Omega): \operatorname{Tr}_{\Omega,\partial\Omega}w=0 \quad \text{$\mu$-a.e. on $\Gamma_{\mathrm{Dir}}$}\}
\end{equation}
is a closed subspace of $W^{1,2}(\Omega)$. Accordingly, the image $\operatorname{Tr}_{\Omega,\partial\Omega}( V(\Omega,\Gamma_{\mathrm{Dir}}))$ of this space under $\operatorname{Tr}_{\Omega,\partial\Omega}$ is the closed subspace of $\mathrm{Tr}_{\partial\Omega}(W^{1,2}(\mathbb{R}^n))$ consisting of all elements that are zero $\mu$-a.e. on $\Gamma_{\mathrm{Dir}}$.

\begin{corollary}\label{C:Dirichletdata} Let $(\Omega,\mu)$ be a bounded $W^{1,2}$-admissible domain in $\mathbb{R}^n$ and let 
	$\Gamma_{\mathrm{Dir}}$ be a Borel subset of $\partial\Omega$ with $\mu(\Gamma_{\mathrm{Dir}})>0$. 
	\begin{enumerate}
		\item[(i)] The Poincar\'e inequality 
		\begin{equation}\label{E:Poinc}
		\int_\Omega |u|^2 dx\leq C_P(\Omega, \mu, \Gamma_{\mathrm{Dir}})\int_\Omega |\nabla u|^2dx,\quad u\in  V(\Omega,\Gamma_{\mathrm{Dir}}),
		\end{equation}
		holds with a constant $C_P(\Omega,\mu, \Gamma_{\mathrm{Dir}})>0$. 
		\item[(ii)] For all $u\in V(\Omega, \Gamma_{\mathrm{Dir}})$ with $\Delta u\in L^2(\Omega)$ we can define a bounded linear functional $\frac{\partial u}{\partial n}\in (\operatorname{Tr}_{\Omega,\partial\Omega}( V(\Omega,\Gamma_{\mathrm{Dir}})))'$ by a counterpart of ~\eqref{FracGreen} when testing with functions $v\in  V(\Omega,\Gamma_{\mathrm{Dir}})$.
		\item[(iii)] Suppose $\gamma$ is a nonnegative and bounded Borel function on $\partial\Omega$ which is positive on a set of positive $\mu$-measure. Then for any $\varphi\in \mathrm{Tr}_{\partial\Omega}(W^{1,2}(\mathbb{R}^n))$ there is a function $\varphi_{\gamma,\bot} \in \mathrm{Tr}_{\partial\Omega}(W^{1,2}(\mathbb{R}^n))$ such that 
		\[\left\langle \varphi_{\gamma,\bot}, \operatorname{Tr}_{\Omega,\partial\Omega}v\right\rangle_{\mathrm{Tr}_{\partial\Omega}(W^{1,2}(\mathbb{R}^n)),\gamma}=0,\quad v\in V(\Omega,\Gamma_{\mathrm{Dir}}),\] 
		and $\varphi_{\gamma,\bot}=\varphi$ $\mu$-a.e. on $\Gamma_{\mathrm{Dir}}$. Here notation is as in ~\eqref{E:altspB}. 
	\end{enumerate}
\end{corollary}
\begin{proof}
	The proof of ~\eqref{E:Poinc} is standard, see for instance \cite[Proposition 7.1]{Egert2015}, the second statement follows like ~\eqref{FracGreen}, and in the third we can take $\varphi_{\gamma,\bot}$ to be the orthogonal projection in $(\mathrm{Tr}_{\partial\Omega}(W^{1,2}(\mathbb{R}^n)), \left\langle\cdot,\cdot\right\rangle_{\mathrm{Tr}_{\partial\Omega}(W^{1,2}(\mathbb{R}^n),\gamma})$ onto the orthogonal complement of $\operatorname{Tr}_{\Omega,\partial\Omega}( V(\Omega,\Gamma_{\mathrm{Dir}}))$.
\end{proof}

\section{Mosco convergence of energy functionals}\label{S:Mosco} 

We consider energy functionals and prove their Mosco convergence, \cite{MOSCO94}, along a convergent sequence of domains. As always, we assume $n\geq 2$.

Suppose that $A$, $B$ and $C$ are positive constants, $D\subset \mathbb{R}^n$ is a bounded Lipschitz domain, $\Omega$ an $(\varepsilon,\infty)$-domain contained in $D$ and $\mu$ is a finite Borel measure with $\Gamma=\supp\mu\subset \overline{\Omega}$ and satisfying ~\eqref{E:upreg} with $n-2<d\leq n$. We define an \emph{energy functional}  $J(\Omega,\mu)$ on $L^2(D)$ by  
\begin{multline}\label{extJ--} 
J(\Omega,\mu)(v)  
=\begin{cases}
A\int\limits_\Omega |v|^2\dx+B\int\limits_\Omega |\nabla v|^2 \dx+C\int\limits_{\Gamma} |\operatorname{Tr}_{\Omega, \Gamma} 
v|^2 d\mu, \ v|_{\Omega}\in W^{1,2}(\Omega) ,\\
+\infty  , \ \ \ v|_{\Omega}\notin W^{1,2}(\Omega).
\end{cases}
\end{multline} 
\begin{remark}\label{R:sectionsareconnected} 
	If $\Gamma\subset \partial\Omega$, then ~\eqref{FracGreen} implies that $J(\Omega,\mu )$ is minimized by the weak solutions $v$, in the sense of testing with elements of $W^{1,2}(\Omega)$, of the Robin problem $B\Delta v = A v$ in $\Omega$ and $B\frac{\partial v}{\partial n} +C\mathds{1}_{\Gamma} \operatorname{Tr}_{\Omega, \partial\Omega} 
	v=0$ on $\partial\Omega$, cf. \cite[Section 22.2g]{Zeidler} or also \cite{CAPITANELLI-2010}. In the next section we will discuss a mixed boundary value problem for the Helmholtz equation. In the case of zero Dirichlet and Robin data the (acoustic) energy \eqref{E:remarkenergy} of the solution to this problem is of a form somewhat similar to ~\eqref{extJ--}, which could be viewed as an equivalent inner product on $W^{1,2}(\Omega)$.
\end{remark} 

Recall that a sequence $(I_m)_m$ of quadratic functionals $I_m:L^2(D)\to [0,+\infty]$ \emph{converges to} a quadratic functional \emph{$I:L^2(D)\to [0,+\infty]$ in the sense of Mosco} in $L^2(D)$ if 

\begin{enumerate}
	\item we have $\varliminf_{m\to \infty} I_m(u_m)\ge I(u)$ for every sequence $(u_m)_{m\in \N^*}$ converging weakly to $u$ in $L^2(D)$, 
	\item for every $u\in L^2(D)$ there exists a sequence $(u_m)_{m\in \N}$ converging strongly in $L^2(D)$ such that $\varlimsup_{m\to \infty} I_m(u_m)\le I(u)$, 
\end{enumerate}
see \cite[Definition 2.1.1]{MOSCO94}.

\begin{remark} The convergence of a sequence of quadratic functionals in the sense of Mosco, \cite{MOSCO94}, implies their Gamma-convergence, \cite{Braides}. Originally convergence in the sense of Mosco was formulated for real Hilbert spaces, \cite{MOSCO94}, but the extension to extended real-valued functionals on complex Hilbert spaces is straightforward.
\end{remark}

The main result of this section is the following.

\begin{theorem}\label{T:Mosco} 
	Let $D\subset \mathbb{R}^n$ be a bounded Lipschitz domain and 
	$\varepsilon>0$. Let $\Omega_m\subset D$ be uniformly bounded $(\varepsilon,\infty)$-domains and
	$\mu_m$ finite Borel measures with $\Gamma_m=\supp\mu_m\subset \overline{\Omega}_m$, all satisfying ~\eqref{E:upreg} with $n-1\leq d\leq n$ and the same constant. For each $m$, let $J(\Omega_m,\mu_m)$ be as in ~\eqref{extJ--} but with $\Omega_m$, $\mu_m$ in place of $\Omega$, $\mu$. 

If $\lim_m \Omega_m=\Omega$ in the Hausdorff sense and in the sense of characteristic functions and $\lim_m \mu_m=\mu$ weakly, then we have	
\begin{equation}\label{E:Moscolimit}
		\lim_m J(\Omega_m,\mu_m)=J(\Omega,\mu).
\end{equation}
 in the sense of Mosco.
\end{theorem}

\begin{remark}  
	For shape admissible domains, Definition~\ref{DefShapeAdmis}, this can be combined with Theorem~\ref{T:compact} (ii). 
\end{remark}

\begin{proof} [Proof of Theorem~\ref{T:Mosco}]
	Note that $J(\Omega,\mu )$ is well-defined: By Theorem~\ref{T:epsinftystable} $\Omega\subset D$ is an $(\varepsilon,\infty)$-domain, and $\Gamma:=\supp\mu$ is contained in the Hausdorff limit $\lim_m \Gamma_m$, which by \cite[2.2.3.2 and Theorem 2.2.25]{HENROT-e} is a subset of $\overline{\Omega}$. 
	
	Let $(u_m)_m\subset L^2(D)$ be a sequence converging to $u$ weakly in $L^2(D)$ and $(u_{m_k})_k\subset (u_m)_m$ such that $\varliminf_m J(\Omega_m, \mu_m)(u_m)=\lim_k  J(\Omega_{m_k}, \mu_{m_k} )(u_{m_k})$.
	We will show that
	\begin{equation}\label{Eqliminf-}
	\lim_k  J(\Omega_{m_k}, \mu_{m_k})(u_{m_k})\geq J(\Omega,\mu)(u),
	\end{equation}
	what then implies the first condition in the definition of Mosco convergence. 
	
	We may assume the left hand side of ~\eqref{Eqliminf-} is finite, hence we can find a subsequence, which for simplicity we still denote by $(u_{m_k})_k$, such that  
	$u_{m_k}\in W^{1,2}(\Omega_{m_k})$ for all $k$ and $\sup_k \|u_{m_k}\|_{W^{1,2}(\Omega_{m_k})}<+\infty$. Since $\mathds{1}_{\Omega_m}\to\mathds{1}_{\Omega}$ in $L^2(D)$ as $m\to\infty$ we may assume that $\mathds{1}_{\Omega_{m_k}}\to \mathds{1}_{\Omega}$ $\lambda^n$-a.e. on $D$ as $k\to \infty$. Since all $\Omega_m$ are bounded $(\varepsilon,\infty)$-domains with the same $\varepsilon$, Proposition~\ref{P:extensionfracSobo} and Remark~\ref{R:LukeandJones} ensure the existence of a constant $c_{\ext}>0$ independent of $k$ such that
	\begin{equation}\label{E:unibound}
	\|\ext_{\Omega_{m_k}} u_{m_k}\|_{W^{1,2}(D)}\le c_{\ext}\|u_{m_k}\|_{W^{1,2}(\Omega_{m_k})}.
	\end{equation}
	We endow $L^2(D)\times L^2(D,\mathbb{R}^n)$ with the Hilbert space norm 
	\begin{equation}\label{E:productnorm}
	(v,w)\mapsto \left\|(v,w)\right\|_{A,B}:=\left(A\int_D |v|^2\dx+B\int_D|w|^2\dx\right)^{1/2}.
	\end{equation}
	Then $\left\|v\right\|_{W^{1,2}(D),A,B}:=\left\|(v,\nabla v)\right\|_{A,B}$ is an equivalent Hilbert space norm on $W^{1,2}(D)$. Since by ~\eqref{E:unibound} the sequence 
	\begin{equation}\label{E:thesequence}
	((\ext_{\Omega_{m_k}} u_{m_k},\nabla \ext_{\Omega_{m_k}} u_{m_k})  )_k
	\end{equation}
	is seen to be bounded in $L^2(D)\times L^2(D,\mathbb{R}^n)$ with respect to ~\eqref{E:productnorm},  we may, passing to further subsequences if necessary, assume that $(\ext_{\Omega_{m_k}} u_{m_k})_k$ converges to some $u^*$ weakly in $W^{1,2}(D)$ w.r.t. $\left\|\cdot\right\|_{W^{1,2}(D),A,B}$, and by the Banach-Saks theorem, \cite[Section 38]{RieszSzNagy1956}, the Ces\`aro means $\frac{1}{N}\sum_{k=1}^N  \ext_{\Omega_{m_k}} u_{m_k}$ converge to $u^\ast$ strongly in $W^{1,2}(D)$. We may similarly assume that \eqref{E:thesequence} converges to some $(v^\ast,w^\ast)$ weakly in $L^2(D)\times L^2(D,\mathbb{R}^n)$, what implies weak convergence for the individual factors. Together with the preceding, this shows that $v^\ast=u^\ast$ and $w^\ast=\nabla u^\ast$.
The $\lambda^n$-a.e. convergence of characteristic functions allows to conclude that $\mathds{1}_{\Omega_{m_k}}\ext_{\Omega_{m_k}} u_{m_k}\to \mathds{1}_\Omega u^\ast$ weakly in $L^2(D)$, and we similarly have $\mathds{1}_{\Omega_{m_k}}\ext_{\Omega_{m_k}} u_{m_k}\to \mathds{1}_\Omega u$ weakly in $L^2(D)$ by the initial assumptions on $(u_m)_m$ and $u$. Combining, we see that $u^\ast|_\Omega=u|_\Omega$ $\lambda^n$-a.e. Therefore
	\[\lim_k \:(\mathds{1}_{\Omega_{m_k}}\ext_{\Omega_{m_k}} u_{m_k},\mathds{1}_{\Omega_{m_k}}\nabla \ext_{\Omega_{m_k}} u_{m_k})=(\mathds{1}_{\Omega}u,\mathds{1}_{\Omega}\nabla u)\]
	weakly in $L^2(D)\times L^2(D,\mathbb{R}^n)$ w.r.t. ~\eqref{E:productnorm}, and as a consequence, 
	\begin{equation}\label{E:liminfofintegrals}
	\varliminf_k  \left\lbrace A\int_{\Omega_{m_k}}|u_{m_k}|^2\dx+B\int_{\Omega_{m_k}}|\nabla u_{m_k}|^2\dx \right\rbrace \geq A\int_{\Omega}|u|^2\dx+B\int_{\Omega}|\nabla u|^2\dx.
	\end{equation}
		
	Let $\frac12<\beta<1$. There is a linear extension operator $\ext_D:H^{\beta}(D)\rightarrow H^{\beta}(\mathbb{R}^n)$ such that with $v_{k}:=\ext_{\Omega_{m_k}} u_{m_k}$ we have 
	\begin{equation}\label{E:compactnesstrick}
	\Vert \ext_D v_{k}- \ext_D u^\ast \Vert_{H^{\beta}(\mathbb{R}^n)}\leq c_{\ext,D} \Vert v_{k}-u^\ast\Vert_{H^{\beta}(D)},
	\end{equation}
	with a constant $c_{\ext,D}>0$, as follows from  Proposition~\ref{P:extensionfracSobo}. Since the embedding of $H^1(D)=W^{1,2}(D)$ in $H^{\beta}(D)$ is compact, see for instance \cite[Theorem 2.7]{Triebel2002}, this goes to zero as $k\to \infty$.

	For the remaining proof we write $v$ to denote the $W^{1,2}(\mathbb{R}^n)$-quasi-continuous modification $\widetilde{v}$ (defined as in  ~\eqref{E:traceop}) of a function $v\in W^{1,2}(\mathbb{R}^n)$. Since $n-1\leq d$ we may apply Lemma~\ref{L:traceconvergence}, ~\eqref{E:compactnesstrick} and the fact that $\Gamma_{m_k}\subset \overline{\Omega}_{m_k}$ and $\Gamma\subset \overline{\Omega}$ to obtain 
	\begin{equation}\label{EqTrGoal-}
	\lim_k  \int_{\Gamma_{m_k}}|
	v_{m_k}|^2 d\mu_{m_k} =\int_{\Gamma}|
	\ext_\Omega u^\ast|^2 d\mu =\int_{\Gamma}|
	\ext_\Omega u|^2 d\mu,
	\end{equation}     
	and combining ~\eqref{E:liminfofintegrals} and ~\eqref{EqTrGoal-} we obtain ~\eqref{Eqliminf-}.

	To prove the second condition we may assume, without loss of generality, that $u\in W^{1,2}(\Omega)$. We claim that it follows with $u_m=\ext_\Omega u$ for all $m$ 
	that 
	\[\lim_{m\to\infty} J(\Omega_m, \mu_m)(\ext_\Omega u)=J(\Omega, \mu )(\ext_\Omega u).\]
	By dominated convergence we have 
	\begin{multline}
	\lim_m \left\lbrace A\int_{\Omega_m}|\ext_\Omega u|^2dx+B\int_{\Omega_m}|\nabla \ext_\Omega u|^2dx\right\rbrace \notag\\
	= A\int_{\Omega}|\ext_\Omega u|^2dx+B\int_{\Omega}|\nabla \ext_\Omega u|^2dx.
	\end{multline}
	For the last term let $w\in C^\infty(\overline{D})$ be such that $\|\ext_\Omega u - w\|_{W^{1,2}(\Omega)}<\varepsilon$. Then 
	\begin{multline}
	\left|\int_{\Gamma_{m}}|
	\ext_\Omega u|^2 d\mu_{m} -\int_{\Gamma}|
	\ext_\Omega u|^2  d\mu \right|\notag 
	\leq \left|\int_{\Gamma_{m}}|
	\ext_\Omega u|^2 d\mu_{m} -\int_{\Gamma_m}|w|^2  d\mu_m \right|
	\\+ \left|\int_{\Gamma_{m}}|w|^2 d\mu_{m} -\int_{\Gamma}|w|^2  d\mu \right|\notag 
	+\left|\int_{\Gamma}|w|^2  d\mu -\int_{\Gamma_m}|
	\ext_\Omega u|^2  d\mu \right|,
	\end{multline}
	estimates similar as in the proof of Lemma~\ref{L:traceconvergence} show that the first and the third summand on the right hand side are smaller than a constant times $\varepsilon$, and for large $m$ the second summand is small by 
	weak convergence.  
\end{proof}

\begin{remark}
	Following the same arguments one can obtain versions of Theorem~\ref{T:Mosco} if in ~\eqref{extJ--} the space $W^{1,2}(\Omega)$ is replaced by a suitable subspace of $W^{1,2}(\Omega)$. For instance, one can consider $V(\Omega,\Gamma_{\mathrm{Dir}})$, defined as in ~\eqref{EqVOM}.
\end{remark}

\section{Shape optimization for the Helmholtz  boundary valued problem}\label{S:shapeOp}

We consider a mixed boundary valued problem for the Helmholtz equation. In~\cite{optimal}, this problem was studied for domains with Lipschitz or $d$-set boundaries. Here we first establish the well-posedness of the problem for $W^{1,2}$-admissible domains $\Omega$ and then verify the existence of optimal shapes in a   class of shape admissible domains. 

The domain $\Omega$ models a tunnel or chamber whose walls may contain noise sources and reflective obstacles for the propagating waves. More precisely, we assume that the boundary $\del \Omega$ of $\Omega$ has different parts on which Dirichlet, Neumann, or Robin boundary conditions are prescribed. Dirichlet conditions model noise sources, and homogeneous Neumann boundary conditions model reflecting walls. The Robin boundary condition involves a fixed complex coefficient $\alpha=\alpha(\omega)$, \cite[Theorem~4]{optimal}, and models partial reflection and absorption at an acoustically absorbent wall made of porous material. As in the most commonly known shape optimization problems, the Dirichlet and Neumann parts of the boundary are kept fixed. The question is what shape the absorbent wall must have in order to minimize the total acoustical energy for a fixed source and a fixed frequency $\omega>0$.

To formalize the model, suppose that $(\Omega,\mu)$ is a $W^{1,2}$-admissible domain in $\mathbb{R}^n$, $n\geq 2$, whose boundary $\del \Omega=\supp\mu$ is divided into
three disjoint parts, 
\begin{equation}\label{E:boundarydecomp}
\partial \Omega = \Gamma_{\mathrm{Dir}} \cup \Gamma_{\mathrm{Neu}} \cup \Gamma,
\end{equation}
each a Borel set and of positive measure $\mu$. Here $\Gamma_{\mathrm{Dir}}$ and $\Gamma_{\mathrm{Neu}}$ denote the fixed Dirichlet and Neumann parts, respectively, and $\Gamma$ denotes the Robin part which may vary, \cite{optimal}. See Figure~\ref{FigGD}, page \pageref{FigGD}, for an example. We consider the formal problem
\begin{equation}\label{EqHelmholtz}
\left\{ \begin{array}{l} \triangle u+\omega^2  u=f, \quad \text{on $\Omega$},\\
\mathrm{Tr}\:u=g\quad \hbox{on }\Gamma_{\mathrm{Dir}},\quad
\dfrac{\del u}{\del n} =0\quad \hbox{on }\Gamma_{\mathrm{Neu}},\quad \dfrac{\del u}{\del n} +\alpha(\omega)\mathrm{Tr}\: u=\mathrm{Tr}\:h\quad \hbox{on }\Gamma, \end{array} \right.
\end{equation}
where $\omega>0$, $\alpha$ is a complex-valued function continuous on $\overline{\Omega}$ with a strictly positive real part $\mathrm{Re}(\alpha)>0$, corresponding to the reflection at $\Gamma$, and a strictly negative imaginary part $\mathrm{Im}(\alpha)<0$, corresponding to the absorption at $\Gamma$, $f$ is a function on $\Omega$, $g$ is a function on $\Gamma_{\mathrm{Dir}}$ and $h$ a function on $\Omega$ with well-defined trace $\mathrm{Tr}\:h$ on $\Gamma$. Equation ~\eqref{EqHelmholtz} is a frequency version of a time-dependent wave propagation problem. The case $g=0$ was originally studied in \cite{BARDOS-1994}. See \cite[Section 2]{optimal} for a discussion about how ~\eqref{EqHelmholtz} models the absorption of acoustical energy by a porous wall.

To formulate problem ~\eqref{EqHelmholtz} rigorously, suppose that $(\Omega,\mu)$ is a bounded $W^{1,2}$-admissible domain in $\mathbb{R}^n$ and that $\mu|_{\Gamma_{\mathrm{Dir}}}$ satisfies the hypotheses of Theorem~\ref{ThTrace} with $\Gamma_{\mathrm{Dir}}$ in place of $K$. Given $f\in L_2(\Omega)$, $g\in B_1^{2,2}(\Gamma_{\mathrm{Dir}})$ and $h\in W^{1,2}(\Omega)$,
we call $u\in W^{1,2}(\Omega)$ a \emph{weak solution} of~~\eqref{EqHelmholtz} on $(\Omega,\mu)$ if $\mathrm{Tr}_{\Omega,\partial\Omega}u=g$ $\mu$-a.e. on $\Gamma_{\mathrm{Dir}}$ and 
\begin{multline}\label{eqVFHfh}
\int_{\Omega}\nabla u \nabla \bar{v} \dx  - \omega^2\int_{\Omega}u\bar{v} \dx +\int_{\Gamma}\alpha \, \mathrm{Tr}_{\Omega,\partial\Omega}  u \, \mathrm{Tr}_{\Omega,\partial\Omega}\bar{v} \,d \mu\\
=-\int_{\Omega}f\bar{v} \dx  +\int_{\Gamma}\mathrm{Tr}_{\Omega,\partial\Omega} h \, \mathrm{Tr}_{\Omega,\partial\Omega} \bar{v} \, d \mu
\end{multline}
for all $v\in V(\Omega,\Gamma_{\mathrm{Dir}})$. Note that $\frac{\partial u}{\partial n} \in (B_1^{2,2}(\partial\Omega))'$
for a weak solution $u$ of~~\eqref{EqHelmholtz} by ~\eqref{FracGreen}, and by ~\eqref{eqVFHfh} and Corollary~\ref{C:Dirichletdata} we have $\frac{\partial u}{\partial n}=\mathds{1}_{\Gamma}(\mathrm{Tr}_{\Omega,\partial\Omega} h-\alpha \mathrm{Tr}_{\Omega,\partial\Omega}  u)$, seen as an identity in $(\mathrm{Tr}_{\Omega,\partial\Omega}(V(\Omega, \Gamma_{\mathrm{Dir}})))'$, what encodes both the Neumann condition on $\Gamma_{\mathrm{Neu}}$ and the Robin condition on $\Gamma$ in ~\eqref{EqHelmholtz}. 

The following well-posedness result generalizes \cite[Theorem 2.1]{optimal}. 
%

\begin{theorem}\label{ThWPH} 
	Let $\Omega\subset \R^n$ be a bounded $W^{1,2}$-admissible domain with $\del \Omega=\supp\mu$ being the disjoint union ~\eqref{E:boundarydecomp} of three Borel subsets $ \Gamma_{\mathrm{Dir}}$, $\Gamma_{\mathrm{Neu}}$ and $\Gamma$ of positive $\mu$-measure. Suppose that $\Gamma_{\mathrm{Dir}}$ is compact and $\mu|_{\Gamma_{\mathrm{Dir}}}$ satisfies ~\eqref{EqMuDs} and ~\eqref{EqMuLd} with $\Gamma_{\mathrm{Dir}}$ in place of $K$, that $\Gamma$ has nonempty open interior in $\partial\Omega$ and that it has positive distance to $\Gamma_{\mathrm{Dir}}$. Let $\omega>0$ and let $\alpha\in C(\overline{\Omega})$ be such that $\mathrm{Re}(\alpha)>0$ and $\mathrm{Im}(\alpha)<0$. 
	
	Then for any $f\in L^2(\Omega)$, $g\in B_1^{2,2}(\Gamma_{\mathrm{Dir}})$ and $h\in W^{1,2}(\Omega)$ 
	there is a unique weak solution $u$ of the Helmholtz problem~~\eqref{EqHelmholtz} on $(\Omega,\mu)$. Moreover, there is a constant $C>0$, depending only on $\alpha$, $\omega$ and on $C_P(\Omega,\mu,\Gamma_{\mathrm{Dir}})$ from Corollary~\ref{C:Dirichletdata}, such that
	\begin{equation}\label{EqApriori}
	\|u\|_{W^{1,2}(\Omega)}\le C\left(\|f\|_{L^2(\Omega)}+\|g\|_{B_1^{2,2}(\Gamma_{\mathrm{Dir}})}+\|h\|_{W^{1,2}(\Omega)}\right).
	\end{equation}
	In the case $g=0$ the operator $B: L_2(\Omega)\times V(\Omega,\Gamma_{\mathrm{Dir}}) \to V(\Omega,\Gamma_{\mathrm{Dir}})$, $B(f,h)=u$, where $u$ is the weak solution of~~\eqref{EqHelmholtz}, is a compact linear operator.
\end{theorem}

\begin{remark}
	The compactness of $\Gamma_{\mathrm{Dir}}$ and ~\eqref{EqMuDs}, ~\eqref{EqMuLd} for $\mu|_{\Gamma_{\mathrm{Dir}}}$ can be dropped if $B_1^{2,2}(\Gamma_{\mathrm{Dir}})$ is replaced by the orthogonal complement in $\mathrm{Tr}_{\partial\Omega}(W^{1,2}(\mathbb{R}^n))$ of the closed subspace
	$\mathrm{Tr}_{\Omega,\partial\Omega}(V(\Omega,\Gamma_{\mathrm{Dir}}))$, endowed with the minimal energy norm. 
\end{remark}

Theorem~\ref{ThWPH} follows in the same way as \cite[Theorem 2.1]{optimal}: If $g=0$, then, using the Poincar\'e inequality, Theorem~\ref{ThGENERIC}, the Riesz representation theorem and the Fredholm alternative, one obtains unique weak solutions for $h=0$ and $f=0$, respectively, and their sum is the unique weak solution for notrivial $f$ and $h$. This method uses the Cauchy uniqueness shown in \cite[Theorem 1.2]{DARDE-2010} for Lipschitz boundaries, thanks to Remark~\ref{R:triviallowerbound} (i) and ~\eqref{FracGreen0} the proof carries over. 
The case $g\neq 0$ we can deal with by linear superposition: If $\hat{g}$ is the unique element of $W^{1,2}(\Omega)$ such that $\Delta \hat{g}=0$ in $\Omega$, $\mathrm{Tr}_{\Omega,\partial\Omega}\hat{g} =g$ $\mu$-a.e. on $\Gamma_{\mathrm{Dir}}$, $\frac{\partial\hat{g}}{\partial n}=0$ on $\Gamma_{\mathrm{Neu}}$ and $\frac{\partial\hat{g}}{\partial n}+{\mathrm{Re}}(\alpha)\mathrm{Tr}_{\Omega,\partial\Omega}\hat{g} =0$ on $\Gamma$, then $u$ satisfies ~\eqref{eqVFHfh} with given $f$ and $h$ if and only if $u-\hat{g} \in V(\Omega,\Gamma_{\mathrm{Dir}})$ satisfies ~\eqref{eqVFHfh} with 
\begin{equation}\label{E:shifts}
f-\omega^2 \hat{g} \quad \text{and}\quad h-i\:{\mathrm{Im}}(\alpha) \hat{g}.
\end{equation} 
Note that we can always assume $h$ or $h-i\:{\mathrm{Im}}(\alpha) \hat{g}$ to be zero on $\Gamma_{\mathrm{Dir}}$, otherwise we can multiply with a smooth cut-off function.
The function $\hat{g}$ can be obtained using Corollary~\ref{C:Dirichletdata} (iii): If $\mathring{g}$ is an arbitrary extension of $g$ to an element of $\mathrm{Tr}_{\partial\Omega}(W^{1,2}(\mathbb{R}^n))$ and $\gamma:=\mathds{1}_\Gamma{\mathrm{Re}}(\alpha)$, then $\hat{g}:=H_{\partial\Omega,\Omega}(\mathring{g}_{\gamma,\bot})$ is as stated.


\begin{remark} As a corollary of~Theorem~\ref{ThWPH}, the operator $-\Delta$ associated with the boundary conditions of problem~~\eqref{EqHelmholtz} does not have real eigenvalues.
\end{remark}

The acoustic \emph{energy} associated with the Helmholtz problem ~\eqref{EqHelmholtz} with zero Dirichlet and Robin boundary data $g=0$ and $h=0$, is $\mathcal{E}(\Omega,\mu,u(\Omega, \mu)):=\int_\Omega |u|^2\dx$, where $u$ is the unique weak solution, and for $f=0$ identity ~\eqref{eqVFHfh} allows to rewrite this as
\begin{equation}\label{E:remarkenergy}
\mathcal{E}(\Omega,\mu,u(\Omega, \mu))=\frac{1}{\omega^2}\left(\|\nabla u\|^2_{L^2(\Omega,\mathbb{R}^n)} +\|\sqrt{\mathrm{Re}(\alpha)} \mathrm{Tr}_{\Omega,\partial\Omega}u\|^2_{L^2(\Gamma,\mu)}\right).
\end{equation}

We discuss the shape optimization problem for functionals similar to those introduced in~~\eqref{extJ--}, evaluated for the weak solution of~~\eqref{EqHelmholtz} with suitable given data.  This is more specific than ~\eqref{E:remarkenergy} in the sense that $\alpha$ has to be constant, but more general in the sense that it can have an additional term. 

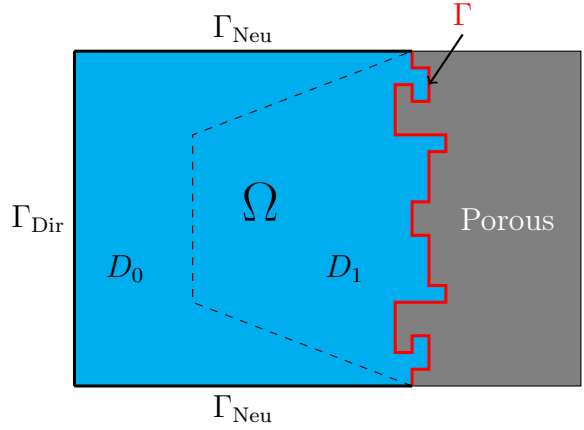
\begin{figure}[!htb]
	\begin{center}
		
		\begin{tikzpicture}[scale=0.222]
		\draw[fill=gray] (0,0) rectangle (30,20);
		\draw[fill=cyan] (20,0)--(0,0)--(0,20)--(20,20)--(20,19)--(21,19)--(21,17)--(20,17)--(20,18)--(19,18)--(19,15)--(22,15)--(22,14)--(21,14)--(21,13)--(21,11)--(20,11)--(20,9)
		--(21,9)--(21,7)--(21,6)--(22,6)--(22,5)--(19,5)--(19,2)--(20,2)--(20,3)--(21,3)--(21,1)--(20,1)--(20,0);
		\draw[red,line width=0.04cm] (20,20)--(20,19)--(21,19)--(21,17)--(20,17)--(20,18)--(19,18)--(19,15)--(22,15)--(22,14)--(21,14)--(21,13)--(21,11)--(20,11)--(20,9)
		--(21,9)--(21,7)--(21,6)--(22,6)--(22,5)--(19,5)--(19,2)--(20,2)--(20,3)--(21,3)--(21,1)--(20,1)--(20,0);
		\draw[thick,->] (23,21)--(21,18);
		\node[above,red] at (23,21) {\large{$\Gamma$}};
		\draw[line width=0.04cm] (0,0)--(20,0);
		\node[below] at (10,0) {\large{$\Gamma_{\mathrm{Neu}}$}};
		\draw[line width=0.04cm] (0,20)--(20,20);
		\node[above] at (10,20) {\large{$\Gamma_{\mathrm{Neu}}$}};
		\draw[line width=0.04cm] (0,0)--(0,20);
		\node[left] at (0,10) {\large{$\Gamma_{\mathrm{Dir}}$}};
		\node at (11,11) {\huge{$\Omega$}};
		\node[left,white] at (29,10) {\large{Porous}};
		\draw [dashed] (20,0) -- (7,5) -- (7,15)  -- (20,20);
		\node at (3,7) {\large{$D_0$}};
		\node at (16,7) {\large{$D_1$}};
		\end{tikzpicture}
		\vspace*{8pt}
		\caption{\label{FigGD} Example of a domain $\Omega$ in $\R^2$, shown in blue, with three types of boundaries: $\Gamma_{\mathrm{Dir}}$ and $\Gamma_{\mathrm{Neu}}$ are fixed and $\Gamma$ can be changed  in the restricted area $\overline{D}_1:=\overline{D}\setminus D_0$. Here $D$ is the large rectangle, which is the union of $\Omega$ filled with the air and of the domain $D\setminus \Omega$ with the porous material. 
		}
	\end{center}
\end{figure}

To define a physically realistic situation, let $D,D_0,D_1\subset \mathbb{R}^n$ be fixed bounded Lipschitz domains, 
such that 
$\overline{D}=\overline{D}_0\cup\overline{D}_1$, 
${D}_0\cap{D}_1=\varnothing$. 
Moreover, we assume that the triple intersection 
$\partial{D}\cap\partial{D}_0\cap\partial{D}_1$ is a $(n-2)$-dimensional Lipschitz sub-manifold of each respective boundary. 
As before we assume that $D_0\subset\Omega\subset D$, and also that $\Gamma_{\mathrm{Dir}}\subset\del D\cap \del D_0$ is a compact non-empty Lipschitz $(n-1)$-dimensional surface disjoint from $\overline{D}_1$, see Figure~\ref{FigGD}. In this set-up we define 
\begin{equation}\label{E:O}
\Gamma:=\del \Omega\cap\overline{D}_1,\qquad\Gamma_{\mathrm{Neu}}:=\del \Omega\setminus (\Gamma\cup \Gamma_{\mathrm{Dir}})=(\del D\cap \del D_0)\setminus   \Gamma_{\mathrm{Dir}}.
\end{equation}
For fixed $\varepsilon>0$, $n-1\leq s<n$ and $n-2<d\leq s$ we write $\hat{U}_{ad}$ for the class of all $(\Omega,\mu) \in U_{ad}(D,D_0,\varepsilon, s, d, \bar c_{s}^A,c_{d}^A)$, where $\Omega$ is as just outlined and $\mu$ is the sum of the $(n-1)$-dimensional Hausdorff measure $\mathcal{H}^{n-1}$ on $\Gamma_{\mathrm{Dir}}\cup \Gamma_{\mathrm{Neu}}$ and a more general measure $\mu_\Gamma$ on $\Gamma$. In Theorem~\ref{ThExistOptimalShape} below we allow $(\Omega,\mu)$ to vary over $\hat{U}_{ad}$, and by the above assumptions this means that we allow $\Gamma$ and $\mu_\Gamma$ to vary.

Let $A\ge 0$, $B\ge 0$ and $C\ge0$. Suppose that 
$(\Omega,\mu)\in \hat{U}_{ad}$, 
$\omega>0$, $\alpha\in C(\overline{D}_1)$ and we are given data $f\in L^2(D)$, $g \in B_1^{2,2}(\Gamma_{\mathrm{Dir}})$ and $h\in W^{1,2}(D_1)$
and
\begin{align} \label{Jen}
& J(\Omega,\mu,u(\Omega,\mu)):=A\int_\Omega |u|^2\dx+B\int_\Omega |\nabla u|^2 \dx+C\int_\Gamma |\mathrm{Tr}_{\Omega,\Gamma}u|^2 d\mu, 
\end{align}
where $u=u(\Omega,\mu)$ denotes the unique weak solution of ~\eqref{eqVFHfh} on $(\Omega,\mu)$ with $f$, $g$, $h$.

\begin{remark}\label{rem-J}  
	To compare to the general form of functionals mentioned in~\cite[p. 156]{HENROT-e}, we point out that one can theoretically consider any objective functional of form
	$$J(\Omega,\mu,u(\Omega, \mu))=\int_\Omega j_1(x,u,\nabla u)\dx+\int_{\del \Omega} j_2(x,\operatorname{Tr}_{\Omega,\Gamma} u)d \mu,$$
	where $j_1: D\times \C\times \C^m \to \R$ is measurable, continuous in $(y,p)$ for almost every $x$ and  such that with a constant $C>0$ we have $|j_1(x,y,p)|\le C(1+|y|^2+|p|^2)$, $x\in D$, $y\in \C$, $p\in \C^n$, and $j_2: \del \Omega \times \C \to \R$ is $\mu$-measurable, continuous in $y$ for almost every $x$ and such that $|j_2(x,y)|\le C(1+|y|^2)$, $x\in \del \Omega$, $y\in \C$.
\end{remark}

We have the following result on the existence of an optimal shape that minimizes $J(\Omega,\mu,u(\Omega,\mu))$ in the class of domains $\hat{U}_{ad}$. 

\begin{theorem}\label{ThExistOptimalShape} Let $\omega>0$ and $\alpha\in C(\overline{D})$. For any $f\in L^2(D)$, $g \in B_1^{2,2}(\Gamma_{\mathrm{Dir}})$ and $h\in W^{1,2}(D)$  there exists an optimal shape $(\Omega_{opt},\mu_{opt})\in \hat{U}_{ad}$
	which minimizes the functional $J(\Omega, \mu, u(\Omega, \mu))$  defined in~\eqref{Jen},
	\begin{equation}\label{E:minimizer}
	J(\Omega_{opt},\mu_{opt},u(\Omega_{opt},\mu_{opt}))=\min_{(\Omega,\mu) \in \hat{U}_{ad}}J(\Omega,\mu,u(\Omega,\mu)).
	\end{equation}
	Moreover, $(\Omega_{opt},\mu_{opt})$ is the limit of a minimizing sequence $(\Omega_m,\mu_m)_{m}\subset \hat{U}_{ad}$ in the Hausdorff sense, the sense of compacts, the sense of characteristic functions and the sense of weak convergence on $\overline{D}$ of the boundary volumes, and the limit $u^\ast=\lim_m \ext_{\Omega_m} u(\Omega_m,\mu_m)$ exists weakly in $W^{1,2}(D)$ and satisfies $u^\ast|_{\Omega_{opt}}=u(\Omega_{opt},\mu_{opt})$.
\end{theorem}

Theorem~\ref{ThExistOptimalShape} follows similarly as \cite[Theorem 3.2]{optimal} by a variational convergence argument:  Theorem~\ref{T:compact} for the domains and Banach-Alaoglu and Lemma~\ref{measure-lemma} for the measures $\mu_\Gamma$ imply the existence of 
a subsequential limit $(\Omega_\ast,\mu_\ast)\in \hat{U}_{ad}$ for a minimizing sequence $(\Omega_m,\mu_m)_{m}\subset \hat{U}_{ad}$. The simultaneous validity of Poincar\'e inequalities with the same constant for all $\Omega_m$ (which follows as in \cite[Theorem 6]{DEKKERS-2020} or, alternatively, by modification of the standard proof as in \cite[Proposition 7.1]{Egert2015} or \cite[Section 5.8]{Evans-2010} together with the convergence the sense of characteristic functions)
implies that the extensions $\ext_{\Omega_m}u_m$ of the unique solutions $u_m$ on the $\Omega_m$ are uniformly bounded in $W^{1,2}(D)$ and therefore have a subsequential weak limit $u^\ast$. Using a variational convergence argument based on ~\eqref{eqVFHfh} and an application of Lemma~\ref{L:traceconvergence} similarly as in the proof of Theorem~\ref{T:Mosco}
one can identify $u^\ast|_{\Omega_{\ast}}$ as the unique weak solution on $(\Omega_\ast,\mu_\ast)$. Using superposition as in ~\eqref{E:shifts}, Corollary~\ref{C:projectright} and Theorem~\ref{ThTracecheap} one can see that similar statements are true for the solutions of the corresponding equations with $g=0$ and shifted data $f$ and $h$, and one can then use ~\eqref{eqVFHfh} for these solutions together with the convergence in the sense of characteristic functions, Rellich-Kondrachov for $D$ and Lemma~\ref{L:traceconvergence} to conclude that $\lim_m J(\Omega_m,\mu_m,u_m)=J(\Omega_\ast, \mu_\ast, u^\ast)$, what shows that $\Omega_{opt}:=\Omega_\ast$ and $\mu_{opt}:=\mu_\ast$ satisfy ~\eqref{E:minimizer}.

\section*{Acknowledgments}

The authors thank 
David Hewett, 
Fr\'ed\'eric Magoul\`es, 
and 
Andrea Moiola 
for interesting and helpful discussions.

\def\refname{References}
\bibliographystyle{siam}
\label{bib:sec}
\bibliography{/home/anna/Documents/Optimization/Fractal/SAIM/2020/biblio.bib}

\begin{thebibliography}{10}

\bibitem{AH96}
{\scshape D.~R. Adams and L.~I. Hedberg}, {\em Function spaces and potential
  theory}, vol.~314 of Grundlehren der Mathematischen Wissenschaften
  [Fundamental Principles of Mathematical Sciences], Springer-Verlag, Berlin,
  1996.

\bibitem{ALEXANDROV-RESHETNYAK}
{\scshape A.~{A}lexandrov and Y.~{R}eshetnyak}, {\em {T}heory of {I}rregular
  {C}urves}, {K}luwer, {D}ordrecht, 1989.

\bibitem{ARFI-2017}
{\scshape K.~Arfi and A.~Rozanova-Pierrat}, {\em Dirichlet-to-{N}eumann or
  {P}oincar\'{e}-{S}teklov operator on fractals described by {$d$}-sets},
  Discrete Contin. Dyn. Syst. Ser. S, 12 (2019), pp.~1--26.

\bibitem{Assouad80}
{\scshape P.~Assouad}, {\em Pseudodistances, facteurs et dimension m\'etrique},
  S\'eminaire d\'{}Analyse Harmonique 1979--1980, Publ. Math. Orsay, 80 (1980),
  pp.~1--33.

\bibitem{AHMNT}
{\scshape J.~Azzam, S.~Hofmann, J.~M. Martell, K.~Nystr\"{o}m, and T.~Toro},
  {\em A new characterization of chord-arc domains}, J. Eur. Math. Soc. (JEMS),
  19 (2017), pp.~967--981.

\bibitem{BARDOS-1994}
{\scshape C.~{B}ardos and J.~{R}auch}, {\em {V}ariational algorithms for the
  {H}elmholtz equation using time evolution and artificial boundaries},
  {A}symptotic {A}nalysis, 9 (1994), pp.~101--117.

\bibitem{Biegert2009}
{\scshape M.~Biegert}, {\em {O}n traces of {S}obolev functions on the boundary
  of extension domains}, Proc. Amer. Math. Soc., 137 (2009), pp.~4169--4176.

\bibitem{Braides}
{\scshape A.~{B}raides}, {\em A handbook of {$\Gamma$}-convergence}, in
  {H}andbook of {D}ifferential {E}quations: {S}tationary {P}artial
  {D}ifferential {E}quations, M.~{C}hipot and P.~{Q}uittner, eds., vol.~3,
  Elsevier, Amsterdam, 2006, ch.~2, pp.~101--213.

\bibitem{BUCUR-2016}
{\scshape D.~{B}ucur and A.~{G}iacomini}, {\em {S}hape optimization problems
  with {R}obin conditions on the free boundary}, {A}nnales de l'{I}nstitut
  {H}enri {P}oincar{\'e} {C}, {A}nalyse non lin{\'e}aire, 33 (2016),
  pp.~1539--1568.

\bibitem{BUCUR-1995}
{\scshape D.~{B}ucur and J.~P. {Z}olesio}, {\em {N}-{D}imensional {S}hape
  {O}ptimization under {C}apacitary {C}onstraint}, {J}. of {D}iff. {E}qus., 123
  (1995), pp.~504--522.

\bibitem{BylundGudayol}
{\scshape P.~Bylund and J.~Gudayol}, {\em On the existence of doubling measures
  with certain regularity properties}, Proc. Amer. Math. Soc., 128 (2000),
  pp.~3317--3327.

\bibitem{CHM2019}
{\scshape A.~Caetano, D.~P. Hewett, and A.~Moiola}, {\em Density results for
  {S}obolev, {B}esov and {T}riebel--{L}izorkin spaces on rough sets},
  arXiv:1904.05420,  (2019).

\bibitem{CALDERON-1961}
{\scshape A.-P. {C}alderon}, {\em {L}ebesgue spaces of differentiable functions
  and distributions}, {P}roc. {S}ymp. {P}ure {M}ath., 4 (1961), pp.~33--49.

\bibitem{CAPITANELLI-2010}
{\scshape R.~{C}apitanelli}, {\em {A}symptotics for mixed {D}irichlet-{R}obin
  problems in irregular domains}, {J}ournal of {M}athematical {A}nalysis and
  {A}pplications, 362 (2010), pp.~450--459.

\bibitem{CAPITANELLI-2010-1}
\leavevmode\vrule height 2pt depth -1.6pt width 23pt, {\em {R}obin boundary
  condition on scale irregular fractals}, {C}ommunications on {P}ure and
  {A}pplied {A}nalysis, 9 (2010), pp.~1221--1234.

\bibitem{C-WH2018}
{\scshape S.~N. Chandler-Wilde and D.~P. Hewett}, {\em Well-posed {PDE} and
  integral equation formulations for scattering by fractal screens}, SIAM J.
  Math. Anal., 50 (2018), pp.~677--717.

\bibitem{C-WHM2017}
{\scshape S.~N. Chandler-Wilde, D.~P. Hewett, and A.~Moiola}, {\em Sobolev
  spaces on non-{L}ipschitz subsets of {$\Bbb{R}^n$} with application to
  boundary integral equations on fractal screens}, Integral Equations Operator
  Theory, 87 (2017), pp.~179--224.

\bibitem{CHENAIS-1975}
{\scshape D.~{C}henais}, {\em {O}n the existence of a solution in a domain
  identification problem}, {J}ournal of {M}athematical {A}nalysis and
  {A}pplications, 52 (1975), pp.~189--219.

\bibitem{DARDE-2010}
{\scshape J.~{D}ard{\'e}}, {\em {M}{\'e}thodes de quasi-r{\'e}versibilit{\'e}
  et de lignes de niveau appliqu{\'e}es aux probl{\`e}mes inverses
  elliptiques}, PhD thesis, Universit\'e {P}aris {D}iderot - {P}aris 7, 2010.

\bibitem{DEKKERS-2020}
{\scshape A.~{D}ekkers, A.~{R}ozanova {P}ierrat, and A.~{T}eplyaev}, {\em
  {M}ixed boundary valued problem for linear and nonlinear wave equations in
  domains with fractal boundaries}, {S}ubmitted. {P}reprint hal-02514311,
  (2020).

\bibitem{DuLiWang}
{\scshape H.~Du, Q.~Li, and C.~Wang}, {\em Compactness of {M}-uniform domains
  and optimal thermal insulation problems}, preprint,  (2020).
\newblock arxiv.org/abs/2008.11144.

\bibitem{Dyn'kin84}
{\scshape E.~M. Dyn{$'$}kin}, {\em Free interpolation by functions with a
  derivative from {$H^{1}$}}, Zap. Nauchn. Sem. Leningrad. Otdel. Mat. Inst.
  Steklov. (LOMI), 126 (1983), pp.~77--87.
\newblock Investigations on linear operators and the theory of functions, XII.
  J. Math. Sci. 27, 2475--2481 (1984).
  \href{https://doi.org/10.1007/BF01474143}{doi.org/10.1007/BF01474143}.

\bibitem{Edgar2008}
{\scshape G.~Edgar}, {\em Measure, {T}opology, and {F}ractal {G}eometry},
  Undergraduate Texts in Mathematics, Springer, New York, 2008.

\bibitem{Egert2015}
{\scshape M.~Egert, R.~Haller-Dintelmann, and J.~Rehberg}, {\em {H}ardy's
  inequality for functions vanishing on a part of the boundary}, Pot. Anal., 43
  (2015), pp.~49--78.

\bibitem{Evans-2010}
{\scshape L.~C. {E}vans}, {\em {P}artial {D}ifferential {E}quations},
  {A}merican {M}ath {S}ociety, 2010.

\bibitem{FALCONER}
{\scshape K.~J. {F}alconer}, {\em {F}ractal {G}eometry - {M}athematical
  {F}oundations and {A}pplications}, John Wiley and Sons, Chichester, 1990.

\bibitem{FarkasJacob2001}
{\scshape W.~Farkas and N.~Jacob}, {\em {S}obolev spaces on non - smooth
  domains and {D}irichlet forms related to subordinate reflecting diffusions},
  Math. Nachr., 224 (2001), pp.~75--104.

\bibitem{FEIREISL-2002-1}
{\scshape E.~{F}eireisl}, {\em {S}hape {O}ptimization in {V}iscous
  {C}ompressible {F}luids}, {A}pplied {M}athematics and {O}ptimization, 47
  (2002), pp.~59--78.

\bibitem{FOT94}
{\scshape M.~Fukushima, Y.~Oshima, and M.~Takeda}, {\em Dirichlet {F}orms and
  {S}ymmetric {M}arkov {P}rocesses}, deGruyter, Berlin, New York, 1994.

\bibitem{GR86}
{\scshape V.~Girault and P.-A. Raviart}, {\em {F}inite {E}lement {M}ethods for
  {N}avier-{S}tokes {E}quations}, Springer, Berlin, 1986.

\bibitem{HAJLASZ-2008}
{\scshape P.~{H}ajłasz, P.~{K}oskela, and H.~{T}uominen}, {\em {S}obolev
  embeddings, extensions and measure density condition}, {J}ournal of
  {F}unctional {A}nalysis, 254 (2008), pp.~1217--1234.

\bibitem{HENROT-e}
{\scshape A.~Henrot and M.~Pierre}, {\em Shape variation and optimization},
  vol.~28 of EMS Tracts in Mathematics, European Mathematical Society (EMS),
  Z\"{u}rich, 2018.
\newblock English version of the French publication with additions and updates.

\bibitem{HurewiczWallman1941}
{\scshape W.~Hurewicz and H.~Wallman}, {\em Dimension {T}heory}, Princeton
  University Press, Princeton, 1941.

\bibitem{JONES-1981}
{\scshape P.~W. {J}ones}, {\em {Q}uasiconformal mappings and extendability of
  functions in {S}obolev spaces}, {A}cta {M}athematica, 147 (1981), pp.~71--88.

\bibitem{JONSSON-1994}
{\scshape A.~{J}onsson}, {\em {B}esov spaces on closed subsets of
  $\mathbb{{R}}^n$}, {T}ransactions of the {A}merican {M}athematical {S}ociety,
  341 (1994), pp.~355--370.

\bibitem{JONSSON-2009}
\leavevmode\vrule height 2pt depth -1.6pt width 23pt, {\em {B}esov spaces on
  closed sets by means of atomic decomposition}, {C}omplex {V}ariables and
  {E}lliptic {E}quations, 54 (2009), pp.~585--611.

\bibitem{JONSSON-1984-1}
{\scshape A.~{J}onsson, P.~{S}j{\"o}gren, and H.~{W}allin}, {\em {H}ardy and
  {L}ipschitz spaces on subsets of $\mathbb{{R}}^n$}, {S}tudia {M}ath., 80
  (1984), pp.~141--166.

\bibitem{JONSSON-1984}
{\scshape A.~{J}onsson and H.~{W}allin}, {\em {F}unction spaces on subsets of
  $\mathbb{{R}}^n$}, {M}ath. {R}eports 2, {P}art 1, {H}arwood {A}cad. {P}ubl.
  {L}ondon, 1984.

\bibitem{JONSSON-1995}
\leavevmode\vrule height 2pt depth -1.6pt width 23pt, {\em {T}he dual of
  {B}esov spaces on fractals}, {S}tudia {M}athematica, 112 (1995),
  pp.~285--300.

\bibitem{LANCIA-2002}
{\scshape M.~R. {L}ancia}, {\em {A} {T}ransmission {P}roblem with a {F}ractal
  {I}nterface}, {Z}eitschrift f{\"u}r {A}nalysis und ihre {A}nwendungen, 21
  (2002), pp.~113--133.

\bibitem{LANCIA-2003}
\leavevmode\vrule height 2pt depth -1.6pt width 23pt, {\em {S}econd order
  transmission problems across a fractal surface}, {R}endiconti, {A}ccademia
  {N}azionale delle {S}cienze detta dei {X}{L}, {M}emoire di {M}athematica e
  {A}pplicazioni, {X}{X}{V}{I}{I} (2003), pp.~191--213.

\bibitem{LANCIA-2010}
{\scshape M.~R. {L}ancia and P.~{V}ernole}, {\em {I}rregular {H}eat {F}low
  {P}roblems}, {S}{I}{A}{M} {J}ournal on {M}athematical {A}nalysis, 42 (2010),
  pp.~1539--1567.

\bibitem{LuukkainenSaksman}
{\scshape J.~Luukkainen and E.~Saksman}, {\em Every complete doubling metric
  space carries a doubling measure}, {P}roc. {A}mer. {M}ath. {S}oc., 126
  (1998), pp.~531--534.

\bibitem{optimalPhys}
{\scshape F.~Magoul{\`e}s, T.~P.~K. Nguyen, P.~Omnes, and A.~Rozanova-Pierrat},
  {\em Fractal boundaries in acoustic energy wave absorption}, in preparation,
  (2020).

\bibitem{optimal}
\leavevmode\vrule height 2pt depth -1.6pt width 23pt, {\em Optimal absorption
  of acoustical waves by a boundary}, to appear in SIAM SICON, Preprint
  hal-01558043,  (2020).

\bibitem{optimalandefficient}
\leavevmode\vrule height 2pt depth -1.6pt width 23pt, {\em Optimal and
  efficient shapes in acoustic boundary absorption}, Preprint hal-02543993,
  (2020).

\bibitem{MARTIO-1979}
{\scshape O.~{M}artio and J.~{S}arvas}, {\em {I}njectivity theorems in plane
  and space}, {A}nnales {A}cademiae {S}cientiarum {F}ennicae {S}eries {A} {I}
  {M}athematica, 4 (1979), pp.~383--401.

\bibitem{MATTILA}
{\scshape P.~{M}attila}, {\em {G}eometry of {S}ets and {M}easures in
  {E}uclidean {S}paces - {F}ractals and {R}ectifiability}, Cambridge University
  Press, Cambridge, 1995.

\bibitem{MattilaMoranRey}
{\scshape P.~Mattila, M.~Mor\'{a}n, and J.-M. Rey}, {\em Dimension of a
  measure}, Studia Math., 142 (2000), pp.~219--233.

\bibitem{MOSCO94}
{\scshape U.~Mosco}, {\em Composite media and asymptotic {D}irichlet forms}, J.
  Funct. Anal., 123 (1994), pp.~368--421.

\bibitem{RieszSzNagy1956}
{\scshape F.~Riesz and B.~Sz.-Nagy}, {\em Functional {A}nalysis}, Blackie and
  Son Ltd., London, 1956.

\bibitem{Rogers}
{\scshape L.~G. Rogers}, {\em Degree-independent {S}obolev extension on locally
  uniform domains}, J. Funct. Anal., 235 (2006), pp.~619--665.

\bibitem{ROZANOVA-PIERRAT-2020}
{\scshape A.~{R}ozanova {P}ierrat}, {\em {G}eneralization of
  {R}ellich-{K}ondrachov theorem and trace compacteness in the framework of
  irregular and fractal boundaries}, Maria Rosaria Lancia and Anna
  Rozanova-Pierrat. Fractals in engineering: Theoretical aspects and Numerical
  approximations, In press, ICIAM 2019 - SEMA SIMAI Springer Series
  Publications,{P}reprint hal-02489325,  (2020).

\bibitem{Rudin}
{\scshape W.~Rudin}, {\em Real and complex analysis}, McGraw-Hill Book Co., New
  York, third~ed., 1987.

\bibitem{STEIN-1970}
{\scshape E.~M. {S}tein}, {\em {S}ingular integrals and differentiability
  properties of functions}, {P}rinceton {U}niversity {P}ress, 1970.

\bibitem{SVERAK-1993}
{\scshape V.~{\u{S}}ver\'{a}k}, {\em {O}n optimal shape design}, {J}. {M}ath.
  {P}ures {A}ppl., 72 (1993), pp.~537--551.

\bibitem{Triebel78}
{\scshape H.~Triebel}, {\em {I}nterpolation {T}heory, {F}unction {S}paces,
  {D}ifferential {O}perators}, vol.~18 of {N}orth-{H}olland {M}athematical
  {L}ibrary, North-Holland, Amsterdam, 1978.

\bibitem{TRIEBEL-1997}
{\scshape H.~{T}riebel}, {\em {F}ractals and {S}pectra. {R}elated to {F}ourier
  {A}nalysis and {F}unction {S}paces}, {B}irkh\"{a}user, 1997.

\bibitem{Triebel2002}
{\scshape H.~Triebel}, {\em {F}unction spaces in {L}ipschitz domains and on
  {L}ipschitz manifolds. {C}haracteristic functions as pointwise multipliers},
  Rev. Mat. Complut., 15 (2002), pp.~475--524.

\bibitem{Vaisala88}
{\scshape J.~V\"ais\"al\"a}, {\em Uniform domains}, Tohoku Journal Math., 40
  (1988), pp.~101--118.

\bibitem{VolbergKonyagin}
{\scshape A.~Vol'berg and S.~Konyagin}, {\em On measures with the doubling
  condition}, {I}zv. {A}kad. {N}auk {SSSR} {S}er. {M}at., 51 (1987),
  pp.~666--675.

\bibitem{WALLIN-1991}
{\scshape H.~{W}allin}, {\em {T}he trace to the boundary of {S}obolev spaces on
  a snowflake}, {M}anuscripta {M}ath, 73 (1991), pp.~117--125.

\bibitem{Zeidler}
{\scshape E.~Zeidler}, {\em {N}onlinear {F}unctional {A}nalysis and its
  {A}pplications {II}/{A}: {L}inear {M}onotone {O}perators}, Springer, New
  York, 1990.

\end{thebibliography}
\end{document}